\documentclass[10pt, reqno]{amsart}

\usepackage{enumerate, amsmath, amsthm, amsfonts, amssymb, xy,  mathrsfs, graphicx, paralist, fancyvrb,mathtools, multicol}
\usepackage[usenames, dvipsnames]{xcolor}
\usepackage[margin=1in]{geometry} 
\usepackage[bookmarks, colorlinks=true, linkcolor=blue, citecolor=blue, urlcolor=blue]{hyperref}

\usepackage{lscape}

\input xy
\xyoption{all}

\usepackage{tikz,tikz-cd}	
\usetikzlibrary{positioning, matrix, shapes}  

\usepackage{titlesec}		
\titleformat{\section}[block]{\large\scshape\bfseries\filcenter}{\thesection.}{1em}{}		
\titleformat{\subsection}[hang]{\large\scshape\bfseries}{\thesubsection}{1em}{}			
\titleformat{\subsubsection}[hang]{\large\scshape\bfseries}{\thesubsubsection}{1em}{}			

\usepackage[titles]{tocloft}								     					
\usepackage{multirow}
\usepackage{array}
\usepackage{booktabs}
\newcolumntype{M}[1]{>{\centering\arraybackslash}m{#1}}
\newcolumntype{N}{@{}m{0pt}@{}}
\usepackage{diagbox}
\usepackage{cancel}

\numberwithin{equation}{subsection}
\newtheorem{theorem}[equation]{Theorem}

\newtheorem{proposition}[equation]{Proposition}
\newtheorem{lemma}[equation]{Lemma}
\newtheorem{corollary}[equation]{Corollary}

\newtheorem{claim}[equation]{Claim}

\theoremstyle{definition}
\newtheorem{rmk}[equation]{Remark}
\newenvironment{remark}[1][]{\begin{rmk}[#1] \pushQED{\qed}}{\popQED \end{rmk}}
\newtheorem{eg}[equation]{Example}
\newenvironment{example}[1][]{\begin{eg}[#1] \pushQED{\qed}}{\popQED \end{eg}}
\newtheorem{defn}[equation]{Definition}
\newenvironment{definition}[1][]{\begin{defn}[#1]\pushQED{\qed}}{\popQED \end{defn}}

\makeatletter
\@addtoreset{equation}{section}
\renewcommand{\thesubsection}{%
  \ifnum\c@subsection<1 \@arabic\c@section
  \else \thesection.\@arabic\c@subsection
  \fi
}
\makeatother

\newcommand{\cB}{\mathcal{B}}

\newcommand{\bF}{\mathbf{F}}

\newcommand{\cG}{\mathcal{G}}
\newcommand{\fG}{\mathfrak{G}}

\newcommand{\sG}{\mathscr{G}}

\newcommand{\cI}{\mathcal{I}}

\newcommand{\cJ}{\mathcal{J}}

\newcommand{\bN}{\mathbf{N}}

\newcommand{\bP}{\mathbf{P}}
\newcommand{\cP}{\mathcal{P}}

\newcommand{\bS}{\mathbf{S}}
\newcommand{\cS}{\mathcal{S}}

\newcommand{\bV}{\mathbf{V}}

\newcommand{\bZ}{\mathbf{Z}}

\newcommand{\bk}{\mathbf{k}}

\newcommand{\bv}{\mathbf{v}}

\newcommand{\bx}{\mathbf{x}}

\newcommand{\by}{\mathbf{y}}

\newcommand{\init}{{\rm init}}

\renewcommand{\phi}{\varphi}
\renewcommand{\emptyset}{\varnothing}

\newcommand{\injects}{\hookrightarrow}

\renewcommand{\tilde}[1]{\widetilde{#1}}

\makeatletter
\def\Ddots{\mathinner{\mkern1mu\raise\p@
\vbox{\kern7\p@\hbox{.}}\mkern2mu
\raise4\p@\hbox{.}\mkern2mu\raise7\p@\hbox{.}\mkern1mu}}
\makeatother

\renewcommand{\hom}{\operatorname{Hom}}

\DeclareMathOperator{\Sym}{Sym}

\DeclareMathOperator{\Tor}{Tor}

\DeclareMathOperator{\sgn}{sgn}

\DeclareMathOperator{\Sec}{Sec}
\DeclareMathOperator{\RL}{\mathbf{RL}}

\newcommand{\GL}{\mathbf{GL}}

\newcommand{\Gr}{\mathbf{Gr}}

\DeclareRobustCommand{\gobblefour}[5]{}

\begin{document}

\title[Syzygies of Secant Ideals of Pl\"ucker-embedded Grassmannians]{Syzygies of Secant Ideals of Pl\"ucker-embedded Grassmannians are Generated in Bounded Degree}
\author{Robert P. Laudone}
\address{Department of Mathematics, University of Wisconsin, Madison}
\email{laudone@wisc.edu \newline \indent {\em URL:} \url{https://www.math.wisc.edu/~laudone/}}
\date{\today}
\thanks{RL was supported by NSF grant DMS-1502553.}
\maketitle

\begin{abstract}
Over a field of characteristic $0$, we prove that for each $r \geq 0$ there exists a constant $C(r)$ so that the prime ideal of the $r$th secant variety of any Pl\"ucker-embedded Grassmannian $\Gr(d,n)$ is generated by polynomials of degree at most $C(r)$, where $C(r)$ is independent of $d$ and $n$. This bounded generation ultimately reduces to proving a poset is noetherian, we develop a new method to do this. We then translate the structure we develop to the language of functor categories to prove the $i$th syzygy module of the coordinate ring of the $r$th secant variety of any Pl\"ucker-embedded Grassmannian $\Gr(d,n)$ is concentrated in degrees bounded by a constant $C(i,r)$, which is again independent of $d$ and $n$.
\end{abstract}

\section{Introduction}

Given a vector space $\bV$ of dimension $n$ over a field $\bk$ of characteristic $0$, recall that $\Gr(d,\bV)$ is the space that parametrizes all dimension $d$ subspaces of $\bV$ called the {\em Grassmannian}. We will omit the choice of $\bV$ and just write $\Gr(d,n)$. A classical result in algebraic geometry realizes $\Gr(d,n)$ as a projective variety via the Pl\"ucker embedding. Specifically, we can define a map $\Gr(d,n) \injects \bP(\bigwedge^d \bk^n)$ as follows. Given a $d$-dimensional subspace spanned by $v_1,\dots,v_d$ in $\Gr(d,n)$ we send
\[
{\rm span}(v_1,\dots,v_d) \mapsto v_1 \wedge \cdots \wedge v_d.
\]
This choice of basis is not unique, but when we apply a change of basis we scale the wedge product by the determinant and so this map is well defined on projective space. The $r$th secant variety of the Pl\"ucker embedding of $\Gr(d,n)$ denoted $\Sec_r(\rho(\Gr(d,n))$, is the Zariski closure in $\bP(\bigwedge^d \bk^n)$ of the set of expressions $\sum_{i=0}^r x_i$ where $x_i$ is in the embedded Grassmannian.  Our convention is that the zeroth secant variety is the original variety. 

Secant varieties have long been a topic of interest in algebraic geometry. Despite this, very little is known about their algebraic structure. Many results about secant varieties focus on the dimension of the space or finding bounds on the degrees of set theoretic generators \cite{DE,DK}. Ideal-theoretic generators are hard to find \cite{MM, LM, LO} and accordingly are not well understood.

Specifically for the Pl\"ucker embedding, a good amount is known about the {\em dimensions} of these secant varieties \cite{CGG, BDdG}. Some set-theoretic results are also known. For example, in \cite{KPRS} the authors prove that all Pl\"ucker embeddings are generated set theoretically by pullbacks of the Klein quadric.

Recently, in \cite{DE} the authors greatly expand the scope of \cite{KPRS} to show that for any fixed $r$, the $r$th secant variety of the Pl\"ucker-embedded $\Gr(d,n)$ is defined set theoretically by polynomials of bounded degree independent of $d$ and $n$. They pose a question at the end of their paper about whether the ideal-theoretic version of their theorem holds. Furthermore, they mention that the ideas present in their paper will not suffice to address the ideal-theoretic version.

The purpose of this paper is to answer this question in the affirmative in characteristic $0$. We ultimately prove the following:

\begin{theorem} \label{TheoremA}
Assume ${\rm char}(\bk) = 0$. For each $r \geq 0$, there is a constant $C(r)$ such that the prime ideal of the $r$th secant variety of the Pl\"ucker-embedded $\Gr(d,n)$, is generated by polynomials of degree $\leq C(r)$, where $C(r)$ does not depend on the choice of $d$ or $n$.
\end{theorem}

This theorem has an immediate corollary resulting from the proof techniques. Exact descriptions of $\ast_g$ and $\cdot$, can be found in \S\ref{3}.

\begin{corollary} \label{CorollaryA}
Assume ${\rm char}(\bk) = 0$. For $r \geq 0$, the equations for the $r$th secant variety of the Pl\"ucker embedding of any $\Gr(d,n)$ can be built out of finitely many equations $f_1,\dots,f_N$ of degree bounded by $C(r)$ under the operations $\ast_g$ and $\cdot$.
\end{corollary}

The main idea in proving Theorem \ref{TheoremA} is to combine all of the ideals of the Pl\"ucker-embedded Grassmannians into a Hopf ring $\cP_\Sigma$ which we define in \S\ref{3}. We then prove noetherianity results with respect to the additional structure on this ring.

Once we define the Hopf ring $\cP_\Sigma$ and show it is noetherian, it is natural to ask if all finitely generated modules over $\cP_\Sigma$ are noetherian. To address this question, we must abstract the structure we develop in proving Theorem \ref{TheoremA} to the language of functor categories as seen in \cite{Sa2,CEF,SS1}. After we transition to this language, we use the new tools available to develop a syzygy theory for Pl\"ucker embedded Grassmannians analogous to the $\Delta$-modules seen in \cite{Sn} and the Veronese theory in \cite{Sa2}, in particular we prove the following:

\begin{theorem} \label{TheoremB}
There is a function $C(i,r)$, depending on $i,r$, but independent of $d,n$, such that the $i$th syzygy module of the coordinate ring of the $r$th secant variety of the Pl\"ucker-embedded $\Gr(d,n)$ is concentrated in degrees bounded by $C(i,r)$.
\end{theorem}

This theorem ultimately encapsulates Theorem \ref{TheoremA} when we take $i = 1$, but the structure we develop to prove Theorem \ref{TheoremA} is crucial in proving Theorem \ref{TheoremB}.

\subsection{Outline of Argument} The proof of Theorem \ref{TheoremA} breaks into the following steps:
\begin{enumerate}
\item For fixed $r \geq 0$, we reduce to considering $\Gr(d,(r+2)d)$ as $d$ varies. This will allow us to bound the degrees of the ideal generators for the $r$th secant varieties of any $\Gr(d,n)$. For this we use  \cite[Proposition 5.7]{MM} as explained in \S\ref{5}.
\item We now consider all values of $d$ via the space $\cP_\Sigma = \bigoplus_{n,d} \Sym^n(\bigwedge^d \bk^{(r+2)d})$, where $\bk$ is a field of characteristic $0$. If $\bV$ is a $(r+2)d$ dimensional vector space over $\bk$, we know $\bV \cong \bk^{(r+2)d}$, so it suffices to consider $\cP_\Sigma$. As in \cite{Sa1}, we observe that there are two products on this space: the usual ``external" product that multiplies outside symmetric powers and a new ``internal" product that multiplies inside exterior powers up to an increasing change of index. We show that subspaces of this space which are ideals for both products are finitely generated. The key insight is that in an infinite antichain of monomials in this space, both $n$ and $d$ cannot be unbounded, this is seen in \S\ref{2}. The internal product involves symmetrizations and so we must assume that the field $\bk$ has characteristic $0$. This step is done in \S\ref{3} with the key preparations in \S\ref{2}.
\item Finally, we notice that the two products are compatible with the standard comultiplication on the symmetric algebra $\cP_\Sigma$. We can define secant varieties in terms of comultiplication. Using this structure, we prove the essential fact the ideal of the $r$th secant variety of the direct sum of all the Pl\"ucker ideals corresponding to $\Gr(d,(r+2)d)$ as $d$ varies, is an ideal in $\cP_\Sigma$ with respect to both products. So using the above, because the $(d,n)$-bigraded component of this ideal corresponds to all degree $n$ polynomials in the $r$th secant variety of the Pl\"ucker embedding of $\Gr(d,(r+2)d)$, we can deduce finite generation. This result is stated in \S\ref{Joins and Secants} with most of the preparation and work done in \S\ref{3}.
\end{enumerate}
In the last two sections, the proof of Theorem \ref{TheoremB} breaks into the following steps:
\begin{enumerate}
\item We translate the structure of $\cP_M$ from \S\ref{2} to the language of functor categories, by developing a category $\cG_M$ whose principal projective generated in degree $(0,0)$ corresponds exactly to $\cP_M$ and whose other morphisms $(d,m) \to (e,n)$ encapsulate multiplication from the $(d,m)$ bigraded piece of $\cP_M$ to the $(e,n)$ bigraded piece. We then use the results from \S\ref{2} to show $\cG_M$ is a Gr\"obner category as defined in \cite{SS1}. The bulk of this is done in \S\ref{6}.
\item We then define a symmetrized version of $\cG_M$ called $\sG_M$ whose principal projective generated in degree $(0,0)$ corresponds to $(\cP_\Sigma)_M$ as seen in \S\ref{3}. At the end of \S\ref{6}, we use the fact that $\cG_M$ is Gr\"obner to prove that every finitely generated $\sG_M$-module is noetherian.
\item With this structure we can study free resolutions of secant ideals of Pl\"ucker embedded Grassmannians. In \S\ref{7}, we find a particular free resolution using the principal projectives in $\sG_M$ which allows us to ultimately deduce Theorem \ref{TheoremB}.
\end{enumerate}

\subsection{Relation to previous work}
\begin{itemize}
\item $\Sec_0(\rho(\Gr(d,n))$ is just the Pl\"ucker-embedded Grassmannian. It is well known that its ideal is generated by quadratic polynomials (the Pl\"ucker equations), so $C(0) = 2$. The case for $d=2$ is well known, in particular the ideal of $\Sec_r(\rho(\Gr(2,n)))$ is generated in degree $r+2$ by sub-Phaffians of size $2r+4$ \cite[\S10]{LO}. This implies a lower bound, $C(r) \geq r+2$, but outside of this we know very little about $C(r)$.
\item As mentioned the Veronese case was addressed in \cite{Sa1}. We address the Pl\"ucker case in this paper. Snowden developed $\Delta$-modules in \cite{Sn} to prove a boundedness result about the syzygies of the Segre embeddings. The question still remains, are the ideals of the secant varieties of the Segre embeddings defined in bounded degree? Can these techniques be used to address the Segre case and ultimately prove results about the syzygies of secant varieties as well?
\item  Rather than look at all Pl\"ucker embeddings of Grassmannians, one can consider all Segre embeddings of products of projective spaces or Veronese embeddings of projective space. If a Segre analogue of these methods can be developed, could it also apply to Segre-Veronese embeddings? 
\item In general, computing the ideals of secant varieties is difficult. We refer the reader to \cite{MM, LO, LW} for references concerning these explicit computations for some cases of the Segre, Veronese and Pl\"ucker embeddings.
\item The idea for showing ideals in $\cP_\Sigma$ are finitely generated was motivated mainly by work in \cite{Sa1}. The underlying idea in most of this work is {\em noetherianity up to symmetry}. For a nice introduction we recommend \cite{D}. Ultimately, one works with a space or object on which a group acts and proves finite generation up to the action of this group. This idea is essential in \cite{SS1,SS2,SS3,NSS}, where the authors explore various manifestations of this idea to prove finite generation results for various representations of categories and twisted commutative algebras. These ideas are also present in \cite{CEF, Sn, DE, DK, Hi, HS, To} where they were used to prove more surprising stability theorems.
\end{itemize}

\subsection{Conventions}
For the most part, $\bk$ will denote a field of characteristic $0$. In \S\ref{2}, this assumption is not necessary and so we let $\bk$ be any commutative noetherian ring, but the assumption is needed in the following sections. We always tensor over $\bk$.

We always denote by $\Sigma_n$ the symmetric group on $n$ letters, and we denote the set $\{1,\dots,n\}$ by $[n]$.

Given a vector space $\bV$, $\bigwedge^d \bV$ denotes its $d$th exterior power. Similarly, $\Sym^d \bV$ denotes the $d$th symmetric power and $\Sym(\bV) = \bigoplus_{d \geq 0} \Sym^d \bV$.
\ \\

\noindent {\large {\bf Acknowledgements.}} I thank Steven Sam for directing me towards this problem, and for his constant guidance and helpful conversations. 

\section{Shuffle-Star Algebra} \label{2}

Fix a commutative noetherian ring $\bk$ (for our purposes taking $\bk$ to be a field suffices, but the general case has the same proof and could be useful in the future).

For a fixed $M \in \bZ_{\geq 0}$, consider the following algebra:
\[
\mathcal{P}_M = \bigoplus_{d,n} (\bigwedge^d \bk^{Md})^{\otimes n}.
\]
In general, we will suppress the subscript $M$, only using it when the value of $M$ will affect the definition or result. A {\bf monomial} of $\cP$ is an element of the form $w_1 \otimes \cdots \otimes w_n$ where each $w_i \in \bigwedge^d k^{Md}$ and $w_i = \alpha (v_{j_1} \wedge \cdots \wedge v_{j_d})$ where $v_{j_i}$ come from the standard basis for $\bk^{Md}$ and $\alpha \in \bk$.

 

We define two multiplication structures on $\cP$. The first is the same as the shuffle product in \cite{Sa1}, we recall it for sake of completeness. Pick a subset $\{i_1 < \cdots < i_n\}$ of $[n+m]$ and let $\{j_1 < \cdots < j_m\}$ be its complement; denote this pair of subsets by $\sigma$ and call it a {\bf split} of $[n+m]$. A split defines a shuffle product
\[
(\bigwedge^d \bk^{Md})^{\otimes n} \cdot_{\sigma} (\bigwedge^d \bk^{Md})^{\otimes m} \to (\bigwedge^d \bk^{Md})^{\otimes (n+m)},
\]
where $(u_1 \otimes \cdots \otimes u_n) \cdot_\sigma (v_1 \otimes \cdots \otimes v_m) = w_1 \otimes \cdots \otimes w_{n+m}$ with $w_{i_k} = u_k$ and $w_{j_k} = v_k$. Whenever we write $f \cdot_\sigma g$, we are implicitly assuming that $\sigma$ is a split of the correct format, otherwise define it to be $0$.

Notice there is an action of $\Sigma_\infty$ on $\cP$ by permuting all possible indices. We recall that for $f \in \cP$, the {\bf width} of $f$, denoted $w(f)$, is the smallest integer $n$ such that for every $\sigma \in \Sigma_\infty$ that fixes $\{1,\dots,n\}$, $\sigma$ also fixes $f$. Every element of $\cP$ must be a finite linear combination of monomials, so only finite many indexes can appear. This means every $f \in \cP$ has finite width, or equivalently that $\cP$ satisfies the finite width condition.

We recall the definition of the {\bf monoid of increasing functions}:
\[
{\rm Inc}(\bN) = \{ \rho: \bN \to \bN \; | \; \forall a<b, \rho(a) < \rho(b)\}.
\]
Since $\cP$ carries an action of $\Sigma_\infty$, there is a natural action of ${\rm Inc}(\bN)$ on $\cP$ as follows. Fix $f \in \cP$, for any $\sigma \in \Sigma_\infty$, $\sigma f$ only depends on $\sigma|_{[w(f)]}$ considering $\sigma$ as a function $\bN \to \bN$. For any $\rho \in {\rm Inc}(\bN)$, there exists some $\sigma \in \Sigma_\infty$ such that $\rho|_{[w(f)]} = \sigma|_{[w(f)]}$, define $\rho f = \sigma f$. The same argument presented in \cite[Pages 6-7]{DK} shows that this gives a well defined action of ${\rm Inc}(\bN)$ on $\cP$.

We define a new product $\ast_g$ where $g \in {\rm Inc}(\bN)$. For monomials we define,
\[
(\bigwedge^d \bk^{Md})^{\otimes n} \ast_g (\bigwedge^e \bk^{Me})^{\otimes n} \to (\bigwedge^{d+e} \bk^{M(e+d)})^{\otimes n}.
\]
as follows. We first require $g([Md]) \subseteq [M(e+d)]$, otherwise we define the product to be zero. Whenever we use this product, we will always implicitly assume that $g$ satisfies this property. Suppose $g([Md]) = \{\alpha_1,\dots,\alpha_{Md}\} \subset [M(e+d)]$ with $g(i) = \alpha_i$, also let $\{\beta_1,\dots,\beta_{Me}\} = [M(e+d)] \setminus g([Md])$. For ease of notation, we define $g^c \in {\rm Inc}(\bN)$ to be $g^c(i) = \beta_i$, we call this the {\em complement of $g$}. Then for monomials,
\begin{align*}
&\left[(v_{i_1} \wedge \cdots \wedge v_{i_d}) \otimes \cdots \otimes (v_{i_{nd-n+1}} \wedge \cdots \wedge v_{i_{nd}}) \right] \ast_g \left[(v_{j_1} \wedge \cdots \wedge v_{j_d}) \otimes \cdots \otimes (v_{j_{nd-n+1}} \wedge \cdots \wedge v_{j_{nd}}) \right] \\
&= g(v_{i_1} \wedge \cdots \wedge v_{i_d}) \wedge g^c(v_{j_1} \wedge \cdots \wedge v_{j_d}) \otimes \cdots \otimes g(v_{i_{nd-n+1}} \wedge \cdots \wedge v_{i_{nd}}) \wedge g^c(v_{j_{nd-n+1}} \wedge \cdots \wedge v_{j_{nd}})\\
&=(v_{\alpha_{i_1}} \wedge \cdots \wedge v_{\alpha_{i_d}} \wedge v_{\beta_{j_1}} \wedge \cdots \wedge v_{\beta_{j_d}})  \otimes \cdots \otimes (v_{\alpha_{i_{nd-n+1}}} \wedge \cdots \wedge v_{\alpha_{i_{nd}}} \wedge  v_{\beta_{j_{nd-n+1}}} \wedge \cdots \wedge v_{\beta_{j_{nd}}})
\end{align*}
Where we view the $v_{\alpha_{i_j}}$ and $v_{\beta_{i_j}}$ as the standard basis vectors in $\bk^{M(d+e)}$. For general $f \in \cP_{d,n}$, $h \in \cP_{e,n}$ and $g \in E$ extend bilinearly.
Extend to the rest of $\cP$ by declaring all other products to be $0$. We first notice a few properties of these products. There is a modified associativity.

\begin{lemma} \label{associative}
Given $f \in \cP_{d,n}$, $b \in \cP_{d,m}$ and $a \in \cP_{e,n+m}$. and a split $\sigma$ of $[n+m]$ and $g \in E$, there exist $p_1,\dots,p_r \in \cP_{e,n}$ and $h_1,\dots,h_r \in \cP_{d+e,m}$ so that,
\[
(f \cdot_\sigma b) \ast_g a = \sum_{i=1}^r h_i \cdot_\sigma (f \ast_{g} p_i).
\]
\end{lemma}

\begin{proof}
Both $\cdot_\sigma$ and $\ast_g$ are bilinear, so assume without loss of generality that both $a$ and $b$ are monomials.  Write $a = \alpha_1 \otimes \cdots \otimes \alpha_{n+m}$ and $b = \beta_1 \otimes \cdots \otimes \beta_m$. Suppose $\sigma$ is the split $\{i_1,\dots,i_m\},\{j_1,\dots,j_n\}$ of $[n+m]$.  Furthermore, suppose $[M(d+e)] \setminus g([Md]) = \{\gamma_1,\dots,\gamma_{Me}\}$, let $\alpha_i'$ be the image of $\alpha_i$ under identifying $\bk^{Me}$ with the subspace of $\bk^{M(e+d)}$ spanned by the standard basis vectors $\{v_{\gamma_i}\}$. So if $\alpha_i = w_{j_1} \wedge \cdots \wedge w_{j_e}$, with $\{w_i\}$ the standard basis for $\bk^{Me}$, then $\alpha_i' = v_{\gamma_{j_1}} \wedge \cdots \wedge v_{\gamma_{j_d}}$.

Taking $h_i = g(\beta_1) \wedge \alpha_{i_1}' \otimes \cdots \otimes g(\beta_m) \wedge \alpha_{i_m}'$ and $p_i = \alpha_{j_1} \otimes \cdots \otimes \alpha_{j_n}$ makes the identity valid.
\end{proof}

We can use the two products to define an ideal in $\cP$ as follows,

\begin{definition}
A homogeneous subspace $I \subset \cP$ is an {\bf ideal} if $f \in I$ implies that $f \ast_g h \in I$ and $h \cdot_\sigma f \in I$ for all $h \in \cP$. A subset of elements of $\cP$ generates an ideal $I$ if $I$ is the smallest ideal that contains the subset.
\end{definition}

With this new language, we get an immediate corollary of Lemma \ref{associative}:

\begin{corollary} \label{basis}
If $f_1,f_2,\dots$ generate an ideal $I$, then every element of $I$ can be written as a sum of elements of the form $h \cdot_\sigma (f_i \ast_g a)$ where $a,h \in \cP$. 
\end{corollary}

We wish to use Gr\"obner methods to prove noetherianity of $\cP$, to do this we will work with monomial ideals. In our context, an ideal is a {\bf monomial ideal} if it has a generating set of monomials. Notice, the product of two monomials under our operations is still a monomial. We will show that all monomial ideals are finitely generated.  We will then use this to show that all ideals in $\cP$ are finitely generated. To do this we must first make some definitions and reformulations.
%
%
Each monomial $m \in \cP_{d,n}$ can be encoded as a {\bf reading list} (RL), $S_{d,n} = (S^1,\dots,S^n)$. Where each $S^i$ is an increasing word of length $d$ created from the finite alphabet $[Md]$. In particular, $|S^i| = d$ with $i = 1,\dots,n$. $S^i$ records the indices in tensor position $i$. The subscript indicates which bigraded piece of $\cP$ the monomial is in. Let $\RL$ denote the set of reading lists.

\begin{example}
To get an idea of what this looks like we give some examples encoding in both directions. A basic example would be the monomial $v_1 \wedge v_2 \otimes v_1 \wedge v_3$ corresponds to,
\[
((1,2),(1,3)).
\]
As a more complicated example consider,
\[
((1,5,6),(1,2,4),(1,2,6)).
\]
This corresponds to the monomial $(v_1 \wedge v_5 \wedge v_6) \otimes (v_1 \wedge v_2 \wedge v_4) \otimes (v_1 \wedge v_2 \wedge v_6)$.
\end{example}

Suppose we have a monomial ideal $J$ of  $\cP$ that is not finitely generated. Then there is an infinite list of monomials $m_1,m_2,\dots$ such that the ideal generated by $m_1,\dots,m_i$ does not contain $m_{i+1}$. This list translates to an infinite list of RLs that are incomparable under the following order. Given a RL $S_{d,n} = (S^1,\dots,S^n)$, this corresponds uniquely to a monomial
\[
v_{S^1} \otimes \cdots \otimes v_{S^n}.
\]
We say that $S_{d,n} \leq T_{e,m}$ if and only if $v_{T^1} \otimes \cdots \otimes v_{T^m}$ is in the ideal generated by $v_{S^1} \otimes \cdots \otimes v_{S^n}$. Call this the {\em monomial order} on RLs. Notice the following:

\begin{lemma} \label{mapequiv}
For RLs, $S_{d,n} \leq T_{e,m}$ in the monomial order is equivalent to the existence of a map $f: S_{d,n} \to T_{e,m}$ with the following properties. If $S_{d,n} = (S^1,\dots,S^n)$ and $T_{e,m} =  (T^1,\dots,T^m)$, $f = (f_1,\dots,f_n)$ where $f_i:S_i \to T_{k_i}$ with the following properties:
\begin{enumerate}
\item $k_1 < k_2 < \cdots < k_n$.
\item $f_i(S_i \cap S_j) = f_j(S_i \cap S_j)$, so they agree on overlap.
\item Each $f_i \in {\rm Inc}(\bN)$.
\end{enumerate}
\end{lemma}

\begin{proof}
This follows easily from definition. The first property means we must map the tensor positions in order. The last property is necessary because we have $g \in {\rm Inc}(\bN)$.
\end{proof}


We will use this equivalence often. We define a new relation on the set $\RL$ $\leq_{RL}$, where $S_{d,n} \leq_{RL} T_{e,m}$ if a map as described in Lemma \ref{mapequiv} exists from $S_{d,n} \to T_{e,m}$. This is easily seen to be a partial order. The above discussion can be summarized as follows:

\begin{corollary} \label{antichain}
An infinite list of monomials $m_1,m_2,\dots$ in $\cP$ such that $m_{i+1}$ is not contained in the ideal generated by $m_1,\dots,m_{i}$ induces an infinite chain of incomparable RLs with respect to $\leq_{RL}$.
\end{corollary}
%

%

Given this infinite antichain of RLs from Corollary \ref{antichain} we will show that either $n$ or $d$ must be bounded.

\begin{lemma} \label{ndbounded}
Given an infinite antichain of reading lists under $\leq_{RL}$,
\[
S_{d_1,n_1},S_{d_2,n_2},\dots
 \]
either the $n_i$ or the $d_i$ must be bounded.
\end{lemma}

\begin{proof}
Suppose this is not the case. Fix $S_{d_1,n_1}$ in this antichain, i.e. the first element. If both $d_j$ and $n_j$ are unbounded, we can assume $d_j > 2Md_1$. For any RL $T_{d_j,n_j} = (T^1,T^2,\dots,T^{n_j})$ with $d_j > 2Md_1$ in our antichain, there are $\binom{d_j}{Md_1}$ unique sub-lists of size $Md_1$ in each list $T^i$ of size $d_j$ in $T_{d_j,n_j}$. There are a total of $\binom{Md_j}{Md_1}$ possible lists of this size that could occur in any $T^i$. Hence if 
\[
n_j > n_1 \tfrac{\binom{Md_j}{Md_1}}{\binom{d_j}{Md_1}},
\]
by the pigeon hole principal we will have at least $n_1$ tensor positions (lists) whose intersection has size greater than $Md_1$. 

To see this, notice if $n_j >  n_1 \tfrac{\binom{Md_j}{Md_1}}{\binom{d_j}{Md_1}}$, for each list of size $d_j$, we have $\binom{d_j}{Md_1}$ different sub-lists of size $Md_1$.  If we have a hole for each list of size $d_1$, there are $\binom{Md_j}{Md_1}$ holes. For each list of side $d_j$, we place $\binom{d_j}{Md_1}$ pigeons into distinct holes. The holes are distinct because there are no repeated numbers. Each of these pigeons represents a tensor position that the list appears in. If there are more than $n_1 \tfrac{\binom{Md_j}{Md_1}}{\binom{d_j}{Md_1}}$ lists, this implies we place more than $n_1 \binom{Md_j}{Md_1}$ pigeons in the $\binom{Md_j}{Md_1}$ holes.  The pigeon hole principal implies that there is one hole with at least $n_1$ distinct elements in it.

Hence there is some list of size $Md_1$ that occurs in at least $n_1$ different tensor positions. Such a list cannot occur twice in the same tensor position because tensor positions cannot contain repeated numbers. As a result, there are at least $n_1$ tensor positions whose intersection contains at least $Md_1$ numbers. 

We will show this lower bound is independent of $d_j$. In particular,
\[
n_1 M^{Md_1} 2^{Md_1-1} > n_1 \tfrac{\binom{Md_j}{Md_1}}{\binom{d_j}{Md_1}},
\]
This is because
\[
 \tfrac{\binom{Md_j}{Md_1}}{\binom{d_j}{Md_1}} = \tfrac{(Md_j)(Md_j-1)\cdots (Md_j-Md_1+1)}{(d_j)(d_j-1)\cdots (d_j-Md_1+1)} = \tfrac{M^{Md_1} (d_j)(d_j-\tfrac{1}{M})\cdots (d_j-\tfrac{Md_1}{M}+\tfrac{1}{M})}{(d_j)(d_j-1)\cdots (d_j-Md_1+1)}.
 \]
If we pair each of these terms as $\tfrac{d_j-\tfrac{a}{M}}{d_j - a}$ for $a = 0,\dots,(Md_1-1)$, we see that the largest of the terms is
\[
\tfrac{d_j - \tfrac{Md_1-1}{M}}{d_j-(Md_1-1)}.
\]
This can easily be checked because
\[
\tfrac{d_j-\tfrac{a}{M}}{d_j - a} < \tfrac{d_j-\tfrac{a+1}{M}}{d_j - (a+1)},
\]
reduces to the inequality
\[
\tfrac{d_j}{M} < d_j.
\]
Which is clearly true so long as $d_j$ is positive, which it is. We now claim,
\[
1 \leq \tfrac{d_j - \tfrac{Md_1-1}{M}}{d_j-(Md_1-1)} < 2.
\]
To see this we clear denominators since $d_j > 2Md_1$ we know the denominator is nonzero. So this inequality is equivalent to
\begin{equation} \label{ineq}
d_j - Md_1 + 1 \leq d_j - \tfrac{Md_1}{M} + \tfrac{1}{M} < 2d_j - 2Md_1 + 2.
\end{equation}
The first inequality reduces to
\[
\tfrac{M-1}{M} \leq (M-1) d_1,
\]
which is clearly true. The second inequality reduces to
\[
(2M-1) d_1 -\tfrac{2M-1}{M} < d_j.
\]
However we took $d_j > 2Md_1$, so this is also true. This implies that every term in the product is in the interval $[1,2)$. Clearly every term of the form
\[
\tfrac{d_j-\tfrac{a}{M}}{d_j - a},
\]
is greater than $1$ for $a = 1,\dots,(d_1-1)$ and every term is less than $\tfrac{d_j - \tfrac{d_1-1}{M}}{d_j-(d_1-1)}$ which we showed is less than $2$. Hence 
\[
1 \leq \tfrac{(d_j)(d_j-\tfrac{1}{M})\cdots (d_j-\tfrac{Md_1}{M}+\tfrac{1}{M})}{(d_j)(d_j-1)\cdots (d_j-Md_1+1)} < 2^{Md_1-1}.
\]
Notice this bound applies regardless of $d_j$ so long as $d_j > 2Md_1$. This implies
\[
\tfrac{M^{Md_1} (d_j)(d_j-\tfrac{1}{2})\cdots (d_j-\tfrac{d_1}{2}+\tfrac{1}{2})}{(d_j)(d_j-1)\cdots (d_j-d_1+1)} \leq M^{Md_1} 2^{Md_1-1}.
\]
So if $n_j > n_1 M^{Md_1} 2^{Md_1-1}$ and $d_j > 2Md_1$, we can find $n_1$ tensor positions whose common intersection has size at least $Md_1$. If $n_j$ and $d_j$ are unbounded for any sequence if we fix some $d_1$ and $n_1$, this will always occur. 

Say the $n_1$ tensor positions we find are $i_1,\dots,i_{n_1}$ and $T^{i_1} \cap \cdots \cap T^{i_{n_1}}$ contains the $Md_1$ elements $j_1 < \cdots < j_{Md_1}$. Then we have a clear map $S_{d_1,n_1} \to T_{d_j,n_j}$ where we map $S^n \to T^{i_n}$ and send $i \mapsto j_i$. This is an order preserving injection that satisfies all the properties of Lemma \ref{mapequiv} and implies $S_{d_1,n_1} \leq T_{d_j,n_j}$ which is a contradiction.
\end{proof}


\begin{remark}
We believe the idea in Lemma \ref{ndbounded} could have further applications in showing various posets are noetherian. In particular, the idea is to fix a small element in any given antichain and assume that the size of the elements in this antichain grow arbitrarily. With this assumption we prove the resulting elements are forced to eventually contain a structure that resembles the fixed element. We then deduce that the size of the elements must be bounded in some sense. This often drastically simplifies the problem as we will see below.
\end{remark}

This leaves us with two cases. Either $d_i$ or $n_i$ is bounded. We will show that both lead to a contradiction.

\begin{lemma} \label{nnotbounded}
Given an infinite antichain in $(\RL,\leq_{RL})$,
\[
S_{d_1,n_1}, S_{d_2,n_2},\dots
\]
$n_i$ cannot be bounded.
\end{lemma}

\begin{proof}
Assume $n_i$ is bounded. Then since our antichain is infinite, this implies that there must be some infinite subchain of our antichain with $n_i$ constant. Restrict our attention to this sub-antichain with $n_i = n$.

We will now embed each $S_{d_i,n}$ of this antichain as a labeled tree and derive a contradiction via Kruskal's tree theorem.

Send each $S_{d_i,n}$ to the tree $T_{S_{d_i,n}}$ with a root vertex labeled $(0,0,(n+1,n+1,\dots,n+1))$. Define the function $\psi(j,S_{d_i,n})$, which takes as an input some RL $S_{d_i,n}$ and some element $j \in [Md_i]$ and returns the finite list of size $n$ with a $k$ in position $k$ if $j$ appears in $S^k$ and a zero if $j$ does not appear in $S^k$. This encodes which of the tensor positions $j$ appears in.

There will be $n$ branches off of the root. Branch $j$ will have $Md_i$ vertices. Vertex $k$ of branch $j$ will be labeled by $(k,j,\psi(k,S_{d_i,n}))$.  Order the first label with the standard order on $\bZ_{\geq 0}$. Order the last two labels with the componentwise subsequence order, in this case the quasi-well-order will be equality. This product is a quasi-well-order by Dickson's lemma because the alphabets for the last two labels are finite.  Hence, each component is a quasi-well order and Dickson's lemma tells us that the finite product of quasi-well orders compared componentwise is also a quasi-well order.

Notice that this is an injective mapping from RLs to trees because we can easily recover $S_{d,n}$ from $T_S$ by reading off vertices. 

Furthermore, using the order described in Kruskal's tree theorem if $T_S \leq T_W$, this implies that $S \leq W$.  We must send $S^i$ to $W^i$ because of the second label. Also, for any $k$, $\psi(k,S_{d_i,n})$ is fixed and we have that $T_S \leq T_W$ if every vertex $v$ maps to some vertex $F(v)$ with $v \leq F(v)$. In combination with the first label, this implies that every number $k$ maps to some number $m \geq k$ with $\psi(k,S_{d_i,n}) = \psi(m,S_{d_i,n})$. As a result, in each of the branches where $k$ occurs, there is some number $m \geq k$ that it can map to because $m$ will occur in all of the remaining branches in which $k$ occurs.

Additionally, we must map branches to branches. Hence if a vertex $v$ with first label $k$ maps to a vertex $F(v)$ with first label $m_k$, we have $m_1 < m_2 < \cdots < m_{Md_i}$, so the map on indices is in the monoid of increasing functions.

Define a mapping inductively form $S \to W$. Begin by sending $1$ to the minimal first label $m$ that occurs for all the vertices corresponding to a vertex with first label $1$. That is, let $\{v_1,\dots,v_\ell\}$ be all the vertices in the tree $T_S$ with first entry $1$. 
Each $v_i$ has an image $F(v_i)$ in $T_W$. Out of all the vertices $\{F(v_1),\dots,F(v_\ell)\}$, send $1$ to the minimum first entry that occurs, call it $\alpha_1$. By construction if $1$ occurs in any branch, so must $\alpha_1$. Hence if $1$ occurs in $S^j$ it has a well defined image in the corresponding $W^j$. Put $f(1) = \alpha_1$.

Now repeat the same procedure for $2$. The element we send $2$ to cannot be $\alpha_1$ because even if $1$ and $2$ occur in all the same branches, any vertex corresponding to a $1$ occurs earlier in the branch that the vertex corresponding to a $2$ and we preserve this order. Hence if $\alpha_1$ is the minimal element for $2$, this would imply $\alpha_1$ is not minimal for $1$. Continue in this way until we have a mapping of all the numbers occurring in $S$.

By construction, two numbers could map to the same $\beta_j$ if and only if they occur in exactly the same branches. So $F$ being well defined implies that $f$ is well defined, i.e. that every element has an image.

Furthermore, $f$ is a map of RLs because if a vertex with first entry $j$ is mapped to another vertex with first entry $m$, $m$ must occur in all of the branches that $j$ does. Hence when we map $j$ to $m$, $j$ has an image in each restriction. The map also satisfies property (4) in Lemma \ref{mapequiv} because we must stay within a branch and we can only map a number to another number that is greater than or equal to it. Finally, $S^j$ must map into $W^j$ because the third label on each vertex must be equal.

The contrapositive implies that our infinite antichain of RLs yields an infinite antichain of trees. This contradicts Kruskal's tree theorem.
\end{proof}

\begin{remark}
We use Kruskal's tree theorem because it provides a more intuitive picture, but Higman's lemma \cite[Theorem 1.3]{D} would suffice. We could encode each branch as an element of a quasi-well-ordered set and then view the trees as words of length $n$ over this quasi-well-ordered set. The mapping of trees in this context is equivalent to one of these words being a subsequence of the other.
\end{remark}

\begin{example}
To see this proof in action, consider the following example. Suppose $n = 3$ is fixed and $M = 2$. Given the two trees
\begin{equation}
\xymatrix @R=1pc{
&{(1,1,(1,0,3))} \ar@{-}[r] & {(2,1,(1,2,0))} \ar@{-}[r] &{(3,1,(0,2,0))} \ar@{-}[r] & {(4,1,(0,0,3))}\\
(0,0,(4,4,4)) \ar@{-}[ur]\ar@{-}[r]\ar@{-}[dr] &{(1,2,(1,0,3))} \ar@{-}[r] &(2,2,(1,2,0)) \ar@{-}[r] &(3,2,(0,2,0))\ar@{-}[r]& {(4,2,(0,0,3))}\\
&(1,3,(1,0,3)) \ar@{-}[r]  &(2,3,(1,2,0)) \ar@{-}[r] &(3,3,(0,2,0))\ar@{-}[r]&(4,3,(0,0,3))
}
\end{equation}
and
{\small
\begin{equation}
\xymatrix @C-=0.2cm @R=1pc{
&{(1,1,(1,2,0))} \ar@{-}[r] & {(2,1,(1,0,3))} \ar@{-}[r] & (3,1,(1,2,0)) \ar@{-}[r] &{(4,1,(0,2,0))} \ar@{-}[r] & {(5,1,(0,0,3))} \ar@{-}[r] & (6,1,(0,0,3)) \\
(0,0,(4,4,4)) \ar@{-}[ur]\ar@{-}[r]\ar@{-}[dr] &{(1,2,(1,2,0))} \ar@{-}[r] & {(2,2,(1,0,3))} \ar@{-}[r] & (3,2,(1,2,0)) \ar@{-}[r] &{(4,2,(0,2,0))} \ar@{-}[r] & {(5,2,(0,0,3))} \ar@{-}[r] & (6,2,(0,0,3)) \\
&{(1,3,(1,2,0))} \ar@{-}[r] & {(2,3,(1,0,3))} \ar@{-}[r] & (3,3,(1,2,0)) \ar@{-}[r] &{(4,3,(0,2,0))} \ar@{-}[r] & {(5,3,(0,0,3))} \ar@{-}[r] & (6,3,(0,0,3)) \\}.
\end{equation}}
Call these $T_1$ and $T_2$. It is easy to read off their corresponding RLs. $T_1$ has RL $S_1 = ((1,2),(2,3),(1,4))$ and $T_2$ has RL $S_2 = ((1,2,3),(1,3,4),(2,5,6))$. Notice $T_1 \leq T_2$ in the order described in Kruskal's tree theorem where we map the vertices as below
\begin{align*}
(1,1,(1,0,3)) \mapsto (2,1,(1,0,3)) \qquad \qquad \qquad (2,1,(1,2,0)) \mapsto (3,1,(1,2,0))\\
(3,1,(0,2,0)) \mapsto (4,1,(0,2,0)) \qquad \qquad \qquad (4,1,(0,0,3)) \mapsto (5,1,(0,0,3))\\
(1,2,(1,0,3)) \mapsto (2,1,(1,0,3)) \qquad \qquad \qquad (2,2,(1,2,0)) \mapsto (3,2,(1,2,0))\\
(3,2,(0,2,0)) \mapsto (4,2,(0,2,0)) \qquad \qquad \qquad (4,2,(0,0,3)) \mapsto (6,2,(0,0,3))\\
(1,3,(1,0,3)) \mapsto (2,3,(1,0,3)) \qquad \qquad \qquad (2,3,(1,2,0)) \mapsto (3,3,(1,2,0))\\
(3,3,(0,2,0)) \mapsto (4,3,(0,2,0)) \qquad \qquad \qquad (4,3,(0,0,3)) \mapsto (5,3,(0,0,3))
\end{align*}
There can be multiple embeddings, but we choose one of them. This gives us a map from $S_1 \to S_2$ inductively as described in the proof. We see that $1$ only maps to a vertex with first label $2$.  So we let $f(1) = 2$.  Now we see that $2$ only maps to vertices with first label $3$. So $f(2) = 3$. Continuing we have $f(3) = 4$.  Now $4$ maps to vertices with different first labels, the set of first labels is $\{5,6\}$.  We then map $4$ to the minimal such label that has not yet been used, so $f(4) = 5$.

This map $f$ induces a map from each component of $S_1$ to each component of $S_2$, so we can easily define $f_i: S_1^i \to S_2^i$ by restriction. Clearly this satisfies property $(1)$ of Lemma \ref{mapequiv}. Furthermore, $f$ satisfies properties (2) and (3) of Lemma \ref{mapequiv} because we defined it via restriction and $f$ satisfies property $(4)$ by construction.

As described above this gives us
\[
[(v_1 \wedge v_2) \otimes (v_2 \wedge v_3) \otimes (v_1 \wedge v_4)] \ast_f (v_1 \otimes v_1 \otimes v_2) = (v_2 \wedge v_3 \wedge v_1) \otimes (v_3 \wedge v_4 \wedge v_1) \otimes (v_2 \wedge v_5 \wedge v_6)
\]
Where $f$ is the map found above.
\end{example}

Lemma \ref{nnotbounded} implies that in our infinite antichain we must have $d_i$ bounded. Furthermore, because this is an infinite chain and only finitely many $d_i$ can occur, we can find an infinite subchain with a fixed $d$.

This implies that for this $d$, we have an infinite antichain in
\[
\bigoplus_{n\geq 0}  \left(\bigwedge^d \bk^{Md} \right)^{\otimes n}.
\]
However there are $\beta_d = \binom{Md}{d}$ basis vectors for $\bigwedge^d \bk^{Md}$, order them in some way as $\{e_1,\dots,e_{\beta_d}\}$. Then every monomial in this algebra is a word in the $e_i$. As there are finitely many basis vectors, Higman's lemma implies that for any infinite sequence of such words, two are comparable. Hence we cannot possibly have an infinite antichain.

This discussion proves the following,

\begin{lemma} \label{dnotbounded}
Given an infinite antichain in $(\RL,\leq_{RL})$,
\[
S_{d_1,n_1},S_{d_2,n_2},\dots
\]
$d_i$ cannot be bounded.
\end{lemma}

The results above cumulatively show,

\begin{theorem} \label{RLNoeth}
The poset $(\RL,\leq_{RL})$ is noetherian.
\end{theorem}

\begin{proof}
Suppose there were an infinite anti-chain of RLs. From Lemma \ref{ndbounded} we know that either the $n$ or $d$ that appear in this list must be bounded otherwise it could not be an antichain. However this cannot be the case via Lemmas \ref{dnotbounded} and \ref{nnotbounded}. So no such antichain exists.
\end{proof}

An immediate consequence of this theorem is the following,

\begin{theorem} \label{monfingen}
All monomial ideals of $\cP$ are finitely generated.
\end{theorem}

\begin{proof}
From Corollary \ref{antichain} if there is an infinite anti-chain of monomials, this leads to an infinite anti-chain of RLs under the order described in Lemma \ref{mapequiv}. This contradicts Theorem \ref{RLNoeth}. Hence all monomial ideals of $\cP$ are finitely generated.
%
\end{proof}

Now place a total ordering $\preceq$ on monomials (ignoring coefficients) of the same bidegree $(d,n)$ as follows. Encode each monomial by $(\bZ^d)^n$ via the $d^{th}$ index in the $n^{th}$ tensor position. First, define $\preceq$ on $\bZ^{d}$ using lexicographic ordering, i.e. $(a_1,\dots,a_d) \preceq (b_1,\dots,b_d)$ if the first nonzero element of $(b_1 - a_1,\dots,b_d-a_d)$ is positive. Then compare tensors using lexicographic ordering. We only work with bihomogeneous elements, so it is not necessary to find a way to compare elements of different bidegrees. 

\begin{lemma} \label{order}
Let $m,m',n$ be monomials. For any $g \in {\rm Inc}(\bN)$ and $\sigma$ a splitting,
\begin{enumerate}
\item If $m \preceq m'$, then $m \ast_g n \preceq m' \ast_g n$.
\item For any $n$, if $m \preceq m'$ then $n \cdot_\sigma m \preceq n \cdot_\sigma m$.
\end{enumerate}
\end{lemma}

\begin{proof}
The proof of the first claim follows because $g \in {\rm Inc}(\bN)$. The second follows because we never change any indices.
\end{proof}

Given $f \in \cP_{d,n}$ we define ${\rm init}(f)$ as the largest monomial along with its coefficient with respect to $\preceq$ that has a nonzero coefficient in $f$. Given an ideal $I$, let ${\rm init}(I)$ be the $\bk$-span of $\{{\rm init}(f) \; | \; f \in I \; \text{homogeneous} \}$.

%

%
%

\begin{lemma} \label{monideal}
If $I$ is an ideal, then ${\rm init}(I)$ is a monomial ideal.
\end{lemma}

\begin{proof}
If $m \in {\rm init}(I)$, we have $m = {\rm init}(f)$ for $f \in I$. For any monomial $n \in \cP$, we have $n \cdot_\sigma m =  {\rm init}(n \cdot_\sigma f)$ by Lemma \ref{order} part (2) 
%
Furthermore, we claim that $m \ast_g n = {\rm init}(f \ast_g n)$. Using the bilinearity of $\ast_g$, the result follows from Lemma \ref{order} part (1).

So $m \ast_g n, n \cdot_\sigma m \in {\rm init}(I)$.  To generalize to any $h \in \cP$, because $\ast_g$ is bilinear it suffices to work with monomials and the above implies the desired result. Hence, ${\rm init}(I)$ is a monomial ideal.
\end{proof}

\begin{lemma} \label{gen}
If $I \subset J$ are ideals and $\init(I) = \init(J)$, then $I = J$.  In particular, if $f_1,f_2,\dots \in J$ and $\init(f_1),\init(f_2),\dots$ generate $\init(J)$ then $f_1,f_2,\dots$ generate $J$.
\end{lemma}

\begin{proof}
Suppose $I$ does not equal $J$.  Pick $f \in J \setminus I$ where $\init(f)$ is minimal with respect to $\preceq$.  Then because $\init(I) = \init(J)$, we have $\init(f) = \init(f')$ for some $f' \in I$.  But $\init(f-f')$ is strictly smaller than $\init(f)$ and $f-f' \in J \setminus I$. This is a contradiction.

For the other statement, let $I$ be the ideal generated by $f_1,f_2,\dots$.
\end{proof}

\begin{corollary} \label{fingen}
Every ideal of $\cP$ is finitely generated.
\end{corollary}

\begin{proof}
Combine Corollary \ref{monfingen}, Lemma \ref{monideal} and Lemma \ref{gen}.
\end{proof}

\section{Symmetrizing} \label{3}
What we are really interested in is $\bigoplus_{n,d \geq 0} \Sym^n(\bigwedge^d \bk^{Md})$ for fixed $M$. So we must symmetrize.  Assume $\bk$ is a field of characteristic $0$. Much of this section is translating the structure of \cite[\S3]{Sa1} to suit $\cP$. Define
\[
\cP^\Sigma_M = \bigoplus_{n,d \geq 0} \left( (\bigwedge^d \bk^{Md})^{\otimes n}\right)^{\Sigma_n} \qquad \qquad (\cP_\Sigma)_M = \bigoplus_{n,d \geq 0} \left((\bigwedge^d \bk^{Md})^{\otimes n}\right)_{\Sigma_n},
\]
where $\Sigma_n$ acts by permuting the tensor factors, and the superscript and subscript respectively denote taking invariants and coinvariants. We generally suppress the additional $M$ subscript and when it matters explicitly mention which $M$ we are working with, for ease of notation writing just $\cP^\Sigma$ or $\cP_\Sigma$. 

Our internal product $\ast_g$ respects the structure of $\cP^\Sigma$ and so $\cP^\Sigma$ is a subalgebra of $\cP$ with respect to it. Unfortunately the shuffle product does not respect the symmetric invariance and so $\cP^\Sigma$ is not closed under it. We remedy this by defining
\[
f \cdot g = \sum_\sigma f \cdot_\sigma g.
\]
Averaging over all splits produces symmetric elements and so $\cP^\Sigma$ is naturally closed under this new $\cdot$. Additionally, $\cdot$ is both commutative and associative. Both of these algebras are naturally bigraded by $(d,n)$. We denote these bigraded pieces by $\cP^\Sigma_{d,n}$ and $(\cP_{\Sigma})_{d,n}$. We will only consider bi-homogeneous subspaces of $\cP^\Sigma$ and $\cP_\Sigma$.

For each $d,n$ define a linear projection
\begin{align*}
\pi \colon \cP_{d,n} &\to \cP^\Sigma_{d,n}\\
w_1 \otimes \cdots \otimes w_n &\mapsto \frac{1}{n!} \sum_{\sigma \in \Sigma_n} w_{\sigma(1)} \otimes \cdots \otimes w_{\sigma(n)}.
\end{align*}
If $f \in \cP^\Sigma_{d,n}$, we have $\pi(f) = f$, so $\pi$ is surjective. We let $\pi' = n! \pi$. We also denote the direct sum of these maps by $\pi \colon \cP \to \cP^\Sigma$ and $\pi' \colon \cP \to \cP^\Sigma$. Next, define
\begin{align*}
\fG \colon (\cP_\Sigma)_{d,n} &\to \cP^\Sigma_{d,n}\\
w_1 \cdots w_n &\mapsto \tfrac{1}{n!} \sum_{\sigma \in \Sigma_n} w_{\sigma(1)} \otimes \cdots \otimes w_{\sigma(n)}.
\end{align*}
Then $\fG$ is a linear isomorphism, since $\fG$ is the inverse of the composition
\[
\cP^\Sigma_{d,n} \to \cP_{d,n} \to (\cP_\Sigma)_{d,n},
\]
where the first map is the natural injection and the second is the natural projection. Denote the direct sum of these maps by $\fG \colon \cP_\Sigma \to \cP^\Sigma$. 

\subsection{Properties of $\cP^\Sigma$} 

\begin{lemma} \label{projprop}
\begin{enumerate}
\item If $f \in \cP^\Sigma_{d,n}$ and $h \in \cP_{e,n}$ are homogeneous, then $\pi(f \ast_g h) = f \ast_g \pi(h)$ and $\pi(h \ast_g f) = \pi(h) \ast_g f$.
\item If $f \in \cP_{d,n}$ and $g \in \cP_{d,m}$, then $\binom{n+m}{n} \pi(f \cdot_\sigma g) = \pi(f) \cdot \pi(g)$ for any split of $[n+m]$.
\end{enumerate}
\end{lemma}

\begin{proof}
\begin{enumerate}
\item Both $\pi(f \ast_g h)$ and $f \ast_g \pi(h)$ are bilinear in $h$ and $f$, so without loss we assume $h = h_1 \otimes \cdots \otimes h_n$ for $h_i \in \bigwedge^e \bk^{Me}$ and that $f = \sum_{\sigma \in \Sigma_n} f_{\sigma(1)} \otimes \cdots \otimes f_{\sigma(n)}$ for $f_i \in \bigwedge^d \bk^{Md}$. Then
\begin{align*}
f \ast_g \pi(h) &= \frac{1}{n!} \sum_{\sigma,\tau \in \Sigma_n} gf_{\sigma(1)} \wedge g^c h_{\tau(1)} \otimes \cdots \otimes g f_{\sigma(n)} \wedge g^c h_{\tau(n)}\\
\pi(f \ast_g h) &= \frac{1}{n!} \sum_{\sigma,\tau \in \Sigma_n} gf_{\sigma\tau(1)} \wedge g^c h_{\tau(1)} \otimes \cdots \otimes gf_{\sigma\tau(n)} \wedge g^c h_{\tau(n)}.
\end{align*}
These sums are identical. In the second, perform the change of variables $\sigma \mapsto \sigma\tau^{-1}$. This shows $\pi(f \ast_g h) = f \ast_g \pi(h)$.

We can apply similar reasoning to the second claim by considering
\begin{align*}
\pi(h) \ast_g f &= \frac{1}{n!} \sum_{\sigma,\tau \in \Sigma_n} g h_{\tau(1)} \wedge g^c f_{\sigma(1)} \otimes \cdots \otimes g h_{\tau(n)} \wedge g^cf_{\sigma(n)}\\
\pi(h \ast_g f) &= \frac{1}{n!} \sum_{\sigma,\tau \in \Sigma_n} g h_{\tau(1)} \wedge g^cf_{\sigma\tau(1)} \otimes \cdots \otimes gh_{\tau(n)} \otimes g^cf_{\sigma\tau(n)}.
\end{align*}

\item This is the same proof as in \cite[Lemma 3.1]{Sa1} because it is the same product.
\end{enumerate}
\end{proof}

\begin{definition}
A homogeneous subspace $I \subseteq \cP^\Sigma$ is an {\bf ideal} if $f \in I$ implies that $f \cdot h \in I$ and $f \ast_g h \in I$ for all $h \in \cP^\sigma$ and $g \in {\rm Inc}(\bN)$.
\end{definition}

We now pass the noetherianity of $\cP$ to $\cP^\Sigma$.

\begin{proposition} \label{invfingen}
Every ideal of $\cP^\Sigma$ is finitely generated.
\end{proposition}

\begin{proof}
Let $J$ be an ideal of $\cP^\Sigma$. As $\cP^\Sigma$ naturally lies inside $\cP$, we can consider the ideal in $\cP$ generated by $J$, call it $I$. By Corollary \ref{fingen}, $I$ is finitely generated, say by $f_1,\dots,f_N$. As $J$ generates $I$, we may assume that each of the $f_i$ belong to $J$. We now claim that the $f_i$ also generate $J$ as an ideal in $\cP^\Sigma$. By Corollary \ref{basis}, every $f \in J$ can be written as a sum of terms of the form $h \cdot_\sigma (f_i \ast_g a)$ where $h,a \in \cP$, $g \in {\rm Inc}(\bN)$. By Lemma \ref{projprop} we have
\[
\pi(h \cdot_\sigma (f_i \ast_g a)) = \binom{n+m}{n}^{-1} \pi(h) \cdot (f \ast_g \pi(a)),
\]
here $h \in \cP_{d,n}$ and $f \ast_g a \in \cP_{d,m}$. By the surjectivity of $\pi$, we conclude that every element of $J$ can be written as a sum of terms of the form $h' \cdot (f_i \ast_g a')$ where $h',a' \in \cP^\Sigma$ and $g \in {\rm Inc}(\bN)$.
\end{proof}

For fixed $d$, $\bigoplus_n \cP_{d,n}^\Sigma$ is a free divided power algebra under $\cdot$, and hence is freely generated in degree $n = 1$. So we can define a comultiplication $\Delta \colon \cP^\Sigma \to \cP^\Sigma \otimes \cP^\Sigma$ by $w \mapsto 1 \otimes w + w \otimes 1$ when $w \in \cP^\Sigma_{d,1}$ and requiring that it is an algebra homomorphism for $\cdot$.

We ultimately wish to show that the two products $\cdot$ and $\ast_g$ respect this comultiplication structure. To do this, we first extend the products to $\cP^\Sigma \otimes \cP^\Sigma$ componentwise.

\begin{lemma} \label{comultproduct}
Pick $\bx \in \cP_{d,n}^\Sigma$ and $\by \in \cP_{e,m}^\Sigma$, and $\bv \in \cP^\Sigma_{e,n}$. Then
\begin{align*}
\Delta(\by \cdot \bv) &= \Delta(\by) \cdot \Delta(\bv)\\
\Delta(\bx \ast_g \bv) &= \Delta(\bx) \ast_g \Delta(\bv)
\end{align*}
\end{lemma}

\begin{proof}
The first identity follows from the definition of $\Delta$.

For the second identity, we once again notice $\Delta(\bx \ast_g \bv)$ and $\Delta(\bx) \ast_g \Delta(\bv)$ are both bilinear in $\bx$ and $\bv$. As a result, we may assume
\[
\bx = \sum_{\sigma \in \Sigma_n} x_{\sigma(1)} \otimes \cdots \otimes x_{\sigma(n)},
\]
for some $x_1 \otimes \cdots \otimes x_n \in \cP_{d,n}$ and that
\[
\bv = \sum_{\sigma \in \Sigma_n} v_{\sigma(1)} \otimes \cdots \otimes v_{\sigma(n)},
\]
for some $v_1 \otimes \cdots \otimes v_n \in \cP_{e,n}$. Then by definition $\bv = v_1 \cdots v_n$ and $\bx = x_1 \cdots x_n$, here we are using the $\cdot$ product. We defined $\Delta$ so that it is an algebra homomorphism for $\cdot$, so we have
\[
\Delta(\bv) = \Delta(v_1) \cdots \Delta(v_n) = \sum_{S \subseteq [n]} \pi'(v_S) \otimes \pi'(v_{[n]\setminus S})
\]
\[
\Delta(\bx) = \Delta(x_1) \cdots \Delta(x_n) = \sum_{S \subseteq [n]} \pi'(x_S) \otimes \pi'(x_{[n]\setminus S})
\]
where the sum is over all subsets $S = \{s_1 < \cdots < s_j\}$ of $[n]$ and $v_S = v_{s_1} \otimes \cdots \otimes v_{s_j}$ and $x_s = x_{s_1} \otimes \cdots \otimes x_{s_j}$. This gives
\begin{align} \label{eq1}
\Delta(\bx) \ast_g \Delta(\bv) &= \left(\sum_{T \subseteq [n]} \pi'(x_T) \otimes \pi'(x_{[n]\setminus T}) \right) \ast_g \left( \sum_{S \subseteq [n]} \pi'(v_S) \otimes \pi'(v_{[n] \setminus S}) \right) \nonumber \\
&= \sum_{S,T \subseteq [n]} \pi'(x_T \ast_g \pi'(v_S)) \otimes \pi'(x_{[n]\setminus T} \ast_g \pi'(v_{[n] \setminus S})).
\end{align}
Where in the second equality we use Lemma \ref{projprop}(1) to write $\pi'(x_T) \ast_g \pi'(v_S) = \pi'(x_T \ast_g \pi'(v_S))$.  On the other hand,
\begin{align*}
\bx \ast_g \bv &= \sum_{\sigma,\tau \in \Sigma_n} gx_{\sigma(1)} \wedge g^c v_{\tau(1)} \otimes \cdots \otimes gx_{\sigma(n)} \wedge g^c v_{\tau(n)}\\
&= \sum_{\sigma \in \Sigma_n} (gx_{\sigma(1)} \wedge g^cv_1) \cdots (gx_\sigma(n) \wedge g^c v_n),
\end{align*}
where in the second sum we are using the $\cdot$ product. In particular,
\begin{equation} \label{eq2}
\Delta(\bx \ast_g \bv) = \sum_{\sigma \in \Sigma_n} \sum_{S \subseteq [n]} \pi'(gx \wedge g^cv)_{\sigma,S} \otimes \pi'(gx \wedge g^cv)_{\sigma,[n]\setminus S}
\end{equation}
Where $(gx \wedge g^c v)_{\sigma,S} = gx_{\sigma(s_1)} \wedge g^c v_{s_1} \otimes \cdots gx_{\sigma(s_j)} \wedge g^c v_{s_j}$ if $S = \{s_1 < \cdots < s_j\}$. We also write $[n] \setminus S = \{s_{j+1} < \cdots < s_n\}$.

The same proof as in \cite[Lemma 3.4]{Sa1} shows that the expressions \eqref{eq1} and \eqref{eq2} for $\Delta(\bx) \ast_g \Delta(\bv)$ and $\Delta(\bx \ast_g \bv)$ are equal. 
\end{proof}

We can now present a symmetrized version of Lemma \ref{associative}:

\begin{lemma}
Given $f \in \cP_{d,n}^\Sigma$, $b \in \cP_{d,m}^\Sigma$, and $a \in \cP_{e,n+m}^\Sigma$, we have
\[
(f \cdot b) \ast_g a = (f \otimes b) \ast_g \Delta(a)
\]
\end{lemma}

\begin{proof}
The proof is essentially identical to the proof of Lemma \ref{associative}. Assume that
\[
a = \sum_{\rho \in \Sigma_{n+m}} \alpha_{\rho(1)} \otimes \cdots \otimes \alpha_{\rho(n+m)},
\]
and 
\[
b = \sum_{\tau \in \Sigma_m} \beta_{\tau(1)} \otimes \cdots \otimes \beta_{\tau(m)}.
\]
Write $a^\rho = \alpha_{\rho(1)} \otimes \cdots \otimes \alpha_{\rho(n+m)}$ and $\beta^\tau = \beta_{\tau(1)} \otimes \cdots \beta_{\tau(m)}$. Pick a split $\sigma = \{i_1,\dots,i_n\}, \{j_1,\dots,j_m\}$ of $[n+m]$. Then
\[
(f \cdot_\sigma b^\tau)\ast_g a^\rho  = (f \ast_g (\alpha_{\rho(i_1)} \otimes \cdots \otimes \alpha_{\rho(i_n)})) \cdot_\sigma ((\beta_{\tau(1)} \otimes \cdots \otimes \beta_{\tau(m)}) \ast_g (\alpha_{\rho(i_1)} \otimes \cdots \otimes \alpha_{\rho(i_m)})),
\]
and $(f \cdot b) \ast_g a$ is the sum of these over all choices of $\rho,\tau,\sigma$. This is exactly $(f \otimes b) \ast_g \Delta(a)$.
\end{proof}

\begin{remark}
Expanding on this componentwise definition of our products in $\cP^\Sigma \otimes \cP^\Sigma$, if we write $\Delta(a) = \sum a^{(1)} \otimes a^{(2)}$. Then
\[
(f \cdot b) \ast_g a = (f \otimes b) \ast_g \Delta(a) = \sum (f \ast_g a^{(1)}) \cdot (b \ast_g a^{(2)}).
\]
\end{remark}

\subsection{Properties of $\cP_\Sigma$} We are ultimately interested in $\cP_\Sigma$. It is much easier to work directly with $\cP^\Sigma$ and most of this subsection will be devoted to transferring results from $\cP^\Sigma$ to $\cP_\Sigma$ via $\fG$. 

First we notice that $\cP_\Sigma$ has an algebra structure. We can define $f \cdot h$ to be the image of $\tilde{f} \cdot_\sigma \tilde{h}$ under $B \to B_\Sigma$ for any split $\sigma$ and lifts $\tilde{f},\tilde{h} \in \cP$ of $f,h \in \cP_\Sigma$. It is not hard to see that this is independent of the choice of lists and the choice of split. To define a $\ast_g$-product on $\cP_\Sigma$ we must rely heavily on $\cP^\Sigma$ and the fact that $\fG$ is a linear isomorphism. We define
\begin{equation} \label{ast}
f \ast_g h = \fG^{-1}(\fG(f) \ast_g \fG(h)).
\end{equation}

The reason for constructing this algebra is to prove a bounded generation result about the secant ideals of Pl\"ucker-embedded Grassmannians. We will show how to get this bounded generation from the finite generation in $\cP_\Sigma$ for the sum of the Pl\"ucker ideals as a consequence of what we have developed. This example is complicated and skipping it will not detract from ones understanding of the paper. We include it to both explicitly illustrate the techniques we are developing and demonstrate the need for these more general techniques due to the difficulty of explicitly working through the easiest possible case, i.e, the $0$th secant ideal.

\begin{example} \label{pluckerfingen}
First, we have the Pieri decomposition,
\[
\bigwedge^p \otimes \bigwedge^p = \bigoplus_{i=0}^p \bS_{(2^i,1^{2(p-i)})}.
\]
The symmetric square has a simple description,
\[
\Sym^2\left(\bigwedge^p \right) = \bigoplus_{\substack{0 \leq i \leq p, \\ i \equiv p \pmod 2}} \bS_{(2^i,1^{2(p-i))})}.
\]
The Pl\"ucker equations of the Grassmannian $\Gr(p,n)$ span the sum of the representations where $i < p$. Taking H\"odge dual isomorphisms, we can generate all the Pl\"ucker equations from the basic ones $f_1,f_2,\dots,$ where $f_i \in \Sym^2(\bigwedge^{2n})$ is defined as follows,
\[
f_n = \frac{1}{2} \sum_{S \subset [4n], \; |S| = 2n} (-1)^{\sgn(S)}x_Sx_{[4n]\setminus S}.
\]
To show the Pl\"ucker equations are all generated from finitely many equations it suffices to prove that the ideal generated by the $f_i$ in $\cP_\Sigma$ is finitely generated. To see this explicitly we will show all the $f_i$ are generated by $f_1$ under our products. In particular, we claim
\begin{equation} \label{show}
\sum_{g,\sigma} (-1)^{\sgn(T)} f_n \ast_g (x_{\sigma(1)\sigma(2)}x_{\sigma(3)\sigma(4)}) = \tfrac{\gamma(n)}{2} f_{n+1}
\end{equation}
Where
\[
\gamma(n) = \frac{\binom{4n}{2n}\cdot 6 \cdot \binom{4n+3}{2n+1}}{\binom{4n+4}{2n+2}}.
\]
The numerator is the total number of terms in the sum. The denominator is the number of terms in $f_{n+1}$. We then divide by $2$ because of $\fG$. It is not hard to see $\gamma(n)$ this is an even integer.

We sum over all $g \in {\rm Inc}(\bN)$, $g|_{[4n]}: [4n] \to [4n+4]$ such that $g(1) = 1$ and only one $g$ acts trivially on $f_n$. We also sum over all $\sigma \in \Sigma_4/(\Sigma_2 \times \Sigma_2)$ permutations such that $\sigma(1) < \sigma(2)$ and $\sigma(3) < \sigma(4)$. Let $(j_1<j_2<j_3<j_4) = [4n+4] \setminus g([4n])$ where $j_k$ is the image of $k$ in $[4n+4]$. In this sum, $T = (g(1),g(2),\dots,g(2n),j_{\sigma(1)},j_{\sigma(2)},g(2n+1),g(n+2),\dots,g(4n),j_{\sigma(3)},j_{\sigma(4)})$ and $\sgn(T)$ is the sign of the permutation that orders $T$.  Tracing through definitions, we find that on monomials
\[
x_{i_1i_2}x_{i_3i_4} \ast_g x_{j_1j_2}x_{j_3j_4} = \tfrac{1}{2} ( x_{g(i_1)g(i_2)j_1j_2} x_{g(i_3)g(i_4)j_3j_4} + x_{g(i_1)g(i_2)j_3j_4} x_{g(i_3)g(i_4)j_1j_2}). 
\]
To see this we explicitly calculate what \eqref{ast} does to monomials $f$ and $h$. First we know that
\[
\fG(x_{i_1i_2}x_{i_3i_4}) = \tfrac{1}{2}(x_{i_1i_2} \otimes x_{i_3i_4} + x_{i_3i_4} \otimes x_{i_1i_2}).
\]
So
\[
\fG(x_{i_1i_2}x_{i_3i_4}) \ast_g \fG(x_{j_1j_2}x_{j_3j_4}) = \tfrac{1}{4} (( x_{i_1i_2} \otimes x_{i_3i_4} + x_{i_3i_4} \otimes x_{i_1i_2}) \ast_g ( x_{j_1j_2} \otimes x_{j_3j_4} + x_{j_3j_4} \otimes x_{j_1j_2})).
\]
Expanding by linearity this becomes
\begin{align*}
&\tfrac{1}{4}\left( x_{g(i_1)g(i_2)j_1j_2} \otimes x_{g(i_3)g(i_4)j_3j_4} + x_{g(i_3)g(i_4)j_3j_4} \otimes x_{g(i_1)g(i_2)j_1j_2} \right) \\
&\quad + \tfrac{1}{4} \left( x_{g(i_1)g(i_2)j_3j_4} \otimes x_{g(i_3)g(i_4)j_1j_2} +  x_{g(i_3)g(i_4)j_1j_2} \otimes x_{g(i_1)g(i_2)j_3j_4} \right).
\end{align*}
Now when we apply $\fG^{-1}$ we notice this is exactly
\[
\tfrac{1}{2} (x_{g(i_1)g(i_2)j_1j_2} x_{g(i_3)g(i_4)j_3j_4} + x_{g(i_1)g(i_2)j_3j_4} x_{g(i_3)g(i_4)j_1j_2}).
\]
Now to justify our claim, we will show explicitly that
\begin{equation} \label{fonesum}
\sum_{g,\sigma} f_1 \ast_g (x_{i_\sigma(1)i_\sigma(2)}x_{i_\sigma(3)i_\sigma(4)}) = 9 f_{2}
\end{equation}
and the proof in the general case is exactly the same but with more indices. Referring to the definition,
\[
f_1 = x_{12}x_{34}- x_{13}x_{24} + x_{14}x_{23}.
\]
We will also drop the $\tfrac{1}{2}$ in the definition of $f_i$ without loss of generality because if we show
\[
2 \cdot \sum_{g,\sigma} (-1)^{\sgn(T)} f_n \ast_g (x_{\sigma(1)\sigma(2)}x_{\sigma(3)\sigma(4)}) = \tfrac{\gamma(n)}{2} (2 \cdot f_{n+1})
\]
this is equivalent to proving \eqref{show}. We will compute an element of the sum in \eqref{fonesum} for one choice of $g$ to illustrate how we would proceed in general. We will then show that every term of $f_2$ appears.  If we fix $g = {\rm id}$ and $\sigma = {\rm id}$. Then $[8] \setminus g([4]) = (5,6,7,8)$. So this element of our summand is
\[
(x_{12}x_{34}- x_{13}x_{24} + x_{14}x_{23}) \ast_g (x_{12}x_{34}).
\]
From what we computed above we know this is
\begin{align*}
\tfrac{1}{2} (x_{1256}x_{3478} + x_{1278}x_{3456} - x_{1356}x_{2478} - x_{1378}x_{2456} +x_{1456}x_{2378} + x_{1478}x_{2356})
\end{align*}
So we see that each term in our summand gives us $6$ unique terms. In what follows we will compute without carrying through the $\tfrac{1}{2}$ and add it back in at the end to simplify the exposition. That is, we will prove 
\[
2 \sum_{g,\sigma} f_1 \ast_g (x_{i_\sigma(1)i_\sigma(2)}x_{i_\sigma(3)i_\sigma(4)}) = 18 f_{2}
\]
and this implies \eqref{fonesum}. So every time we apply $\ast_g$ we will not include the resulting $\tfrac{1}{2}$. We will now show that we can get any term $\beta = x_{i_1i_2i_3i_4}x_{i_5i_6i_7i_8}$ by choosing the correct summand of $f_1$, and choice of $g$ and $\sigma$. We can assume that $i_1 < i_2$ and $i_5 < i_6$. Under these conditions there are three possible orders that can occur if we fix $i_1 = 1$. We could have $i_1 < i_2 < i_5 < i_6$, $i_1 < i_5 < i_2 < i_6$ or $i_1 < i_5 < i_6 < i_2$. This choice will determine which of the monomials in $f_1$ we will use to get $\beta$.

In particular, suppose we have $i_1 < i_5 < i_2 < i_6$, then we will use the term $x_{13}x_{24}$ because we know $g$ is order preserving and so this term that will have this order if we insert the numbers as the first two indices in each monomial. Choose $g$ with
\[
g(1) = i_1, \qquad \qquad g(2) = i_5,  \qquad \qquad g(3) = i_2,  \qquad \qquad g(4) = i_6.
\]
Now we have four numbers remaining $\{j_1<j_2<j_3<j_4\} = [8] \setminus g([4])$. Suppose $\tau$ is the permutation with $\tau(i_3) < \tau(i_4) < \tau(i_7) < \tau(i_8)$. Realize $\tau$ as an element of $\Sigma_4/(\Sigma_2 \times \Sigma_2)$ and let $\sigma = \tau^{-1}$. After we apply $\ast_g$ we know $i \mapsto j_i$. We also have $\tau(i_3) = j_1$, $\tau(i_4) = j_2$, $\tau(i_7) = j_3$ and $\tau(i_8) = j_4$. So $\sigma(i)$ maps to  $\sigma j_i = \sigma \tau(i_k) = i_k$, where the permutations act on the indices.

Then in this summand if we focus on the monomial coming from $x_{13}x_{24}$ with our selected $g$ and $\sigma$, we have
\[
x_{13}x_{24} \ast_g x_{\sigma(1)\sigma(2)}x_{\sigma(3)\sigma(4)} = x_{i_1i_2i_3i_4}x_{i_5i_6i_7i_8}.
\]
We only mention this once more, but here we recall that we have multiplied the entire sum by $2$. Furthermore, notice this term has the correct sign. The monomial $x_{13}x_{24}$ is negative in $f_1$ because the permutation ordering this sequence is odd. The monomial above should have sign $(-1)^{\sgn(\tau)}$ where $\tau$ orders the set $(i_1,i_2,i_3,i_4,i_5,i_6,i_7,i_8)$. By construction this set is $(g(1),g(3),i_3,i_4,g(2),g(4),i_7,i_8)$. In our equation this monomial will have sign
\[
(-1)^{\sgn(T)\sgn(\sigma)+1}.
\]
If we apply the odd permutation swapping $g(2)$ with $g(3)$ we have $\sgn(\tau)+1$ is the sign of the permutation ordering $(g(1),g(2),i_3,i_4,g(3),g(4),i_7,i_8)$. But by construction $i_3 = \sigma(j_1)$ etc where $\{j_1,j_2,j_3,j_4\} = [8] \setminus g([4])$. So $\sgn(\tau)+1$ is the sign of the permutation ordering
\[
(g(1),g(2),\sigma(j_1),\sigma(j_2),g(3),g(4),\sigma(j_3),\sigma(j_4)).
\]
This is precisely $\sgn(T)+1$. So the signs match. Notice this argument will work for every other term as well because we apply the permutation to correct the $g(i)$ terms, this sign will cancel with the sign the monomial has in $f_i$.

So every term that appears in this sum has the correct sign and every term of $f_2$ appears in the sum. It remains to show that each term appears exactly $\gamma(n)$ times. In this case $n =1$, $\gamma(1) = 18$. So we want to show each term appears $18$ times. However, it suffices to show this for one term by symmetry.

%

Suppose $\beta = x_{1457}x_{2368}$. We will show $\beta$ appears $18$ times in our sum. $f_1$ has three terms in it, we consider each separately. If we first consider the term $x_{12}x_{34}$. There are two choices of $g$ and $\sigma$ which yield $\beta$. Recalling that $g(1) = 1$ for all $g$, the two choices are,
\[
g(2) = 4 \qquad g(3) = 6 \qquad g(4) = 8 \qquad \text{and} \qquad g(2) = 5 \qquad g(3) = 6 \qquad g(4) = 8.
\]
For the first map we need $\sigma$ to be the permutation $(1,3)(2,4)$ with $(j_1<j_2<j_3<j_4) = (2,3,5,7)$. For the second $\sigma$ is also the permutation $(1,3)(2,4)$. In both cases 
\[
x_{12}x_{34} \ast_g x_{\sigma(1)\sigma(2)}x_{\sigma(3)\sigma(4)} = x_{1457}x_{2368}.
\]
Now if we focus on the second monomial in $f$, $x_{13}x_{24}$ we will see there are $11$ pairs $(g,\sigma)$ which yield $\beta$. If we fix $g(3) = 4$, then there are $2$ choices for $g(2)$ and $2$ choices for $g(4)$, each paired with the appropriate $\sigma$. Similarly if we fix $g(3) = 5$, there are once again $4$ total options. If $g(3) = 7$, there are now $3$ options for $g(2)$ and only $1$ for $g(4)$ because $g$ must be increasing. This gives us $11$ total options.

Finally if we focus on the monomial $x_{14}x_{23}$, if we fix $g(4) = 4$, then $g(2)$ and $g(3)$ must also be fixed. If $g(4) = 5$, once again we can only have $g(2) = 2$ and $g(3) = 3$. If $g(4) = 7$, then there are $\binom{3}{2}$ options for where to send $2$ and $3$. Hence in this case there are $5$ total ways to get $\beta$.

In all, we found $2+11+5 = 18$ ways $\beta$ could appear in our sum. This shows that
\[
2 \sum_{g,\sigma} f_1 \ast_g (x_{i_\sigma(1)i_\sigma(2)}x_{i_\sigma(3)i_\sigma(4)}) = 18 f_2.
\]
Which is equivalent to \eqref{fonesum}. This same argument generalizes to any $f_n$. We are essentially appending missing terms on in all possible ways to account for the symmetry present in $f_n$. 

In particular, this shows that if $f_n$ is in our ideal, $f_{n+1}$ is as well. So $f_1$ generates all the $f_i$. By the above the Pl\"ucker ideal is finitely generated under the operations we have defined.
\end{example}


\begin{definition}
An ideal $\cJ \subseteq \cP_\Sigma$ is a {\bf di-ideal} if $\cG(\cJ)$ is closed under both $\ast_g$ and $\cdot$.
\end{definition}

For ease of notation we make the following definition

\begin{definition} \label{SumDef}
For any fixed $M\geq 2$, call the sum of the Pl\"ucker ideals corresponding to $\Gr(d,Md)$ as $d$ varies, $\cS_{M}$.  So that $(\cS_M)_{d,n}$ is the space of degree $n$ polynomials in the ideal of $\rho(\Gr(d,Md))$ where $\rho$ is the Pl\"ucker embedding.
\end{definition}

We ultimately wish to show that for any $r \geq 0$, $\cS_r$ is a di-ideal. We will see this can be deduced from the following lemma:

\begin{lemma} \label{coinvideal}
For any fixed $M\geq 2$, $\cS_{M}$, is closed under $\ast_g$ and $\cdot$ in $(\cP_\Sigma)_{M}$.
\end{lemma} 

\begin{proof}

Using Weyman's construction of the Pl\"ucker equations in \cite[Proposition 3.1.2]{We}, any Pl\"ucker equation $f$ in the graded coordinate ring of the Pl\"ucker embedding of $\Gr(d,Md)$ which is contained in $(\cP_\Sigma)_{M}$ will be of the form
\[
f = \sum_{\beta} \sgn(\beta) x_{i_1\dots i_u j_{\beta(1)} \dots j_{\beta(d-u)}} x_{j_{\beta(d-u+1)} \dots j_{\beta(2d-u-v)} l_1\dots l_v}
\]
where we sum over all permutations $\beta$ of $\{1,\dots,d-u-v\}$ such that $\beta(1) < \beta(2) < \cdots < \beta(d-u)$ and $\beta(d-u+1) < \beta(d-u+2) < \cdots < \beta(2d-u-v)$. And we choose $i_1,\dots,i_u, j_1,\dots,j_{2d-u-v},l_1,\dots,l_v$ as distinct elements from $[Md]$.

Each term in $f$ is a product $x_Ix_K$ with $I \not= K$, $|I|=|K| = d$ as seen in \cite[Proposition 3.1.6]{We}. An alternative way to see this is to refer back to the description of the image of the Grassmannian in Example \ref{pluckerfingen}. The Pl\"ucker equations span the sum of the representations 
\[
\Sym^2\left(\bigwedge^d \right) = \bigoplus_{\substack{0 \leq i \leq d, \\ i \equiv d \pmod 2}} \bS_{(2^i,1^{2(r-i))})},
\]
with $i < d$. If $I = K$, this would imply we only use $d$ distinct numbers as indices in $x_Ix_K$. However, this only occurs in the representation $\bS_{2^d}$, which means the corresponding element could not be a Pl\"ucker equation. 

Fix an element $n \in (\cP_\Sigma)_{M}$ of bidegree $(a,m)$. For any Pl\"ucker equation $f$, to get $f \ast_g n$, from Example \ref{pluckerfingen}, we append additional distinct indices to the monomials appearing in $f$. Hence, each term in $f \ast_g n$ consists of $x_Ix_K$ where $I \not= K$, with $|I| = |K| = d+a$.

As in Weyman, we may identify each $x_I$ with the element $e_I^\ast \in \wedge^{d+a} (\bk^{M(d+a)})^\ast$ where $e_i^\ast$ denotes the dual basis element to the standard basis element $e_i$ of $\bk^{M(d+a)}$. 

To show that $f \ast_g n \in J$ it suffices to show that it vanishes on all decomposable elements of $\bk^{M(d+a)}$. To do this consider all of the equations given by $f \ast_g n$ for all $f$ a Pl\"ucker equation, $g \in {\rm Inc}(\bN)$ and $n \in (\cP_\Sigma)_{M}$. The following computation is not that enlightening, but the result is very important so we will state it as a claim and skipping the proof will not take away from the proof of this lemma. 

\begin{claim} \label{Equivariant}
The collection of equations $f \ast_g n$ with $f$ a Pl\"ucker equation, $g \in {\rm Inc}(\bN)$ and $n \in (\cP_\Sigma)_{M}$ of fixed bidegree $(a,m)$ is $\GL(\bk^{M(d+a)})$-invariant. 
\end{claim}

\begin{proof}
Indeed, given $h \in \GL(\bk^{M(d+a)})$, we have
\[
h(f \ast_g n) =  [g^{-1}hg f] \ast_{g} (g')^{-1}hg'n. 
\]
Where $g' \in {\rm Inc}(\bN)$ is the map induced by $g$ which sends $[2a]$ to $[M(d+a)] \setminus g([Md])$ in increasing order. It is easy to check on indices that this is valid. 

The set of Pl\"ucker equations is $\GL$-invariant \cite[Proposition 3.1.2]{We} and we can view $g^{-1}hg \in \GL(\bk^{Md})$, so $(g^{-1}hg) f$ is another Pl\"ucker equation. Also, $(g')^{-1}hg'n \in (\cP_\Sigma)_{M}$ and has the same bidegree $(a,m)$.
\end{proof}

Furthermore, $\GL(\bk^{M(d+a)})$ acts transitively on $\Gr(d+a,M(d+a))$ and so acts transitively on the decomposable elements of $\bigwedge^{d+a}\bk^{M(d+a)}$. As a result, it suffices to show that all the equations $f \ast_g n$ as $f$, $g$ and $n$ vary as described in Claim \ref{Equivariant} vanish on one decomposable element, say $e_{i_1} \wedge \cdots \wedge e_{i_{d+a}}$, where $i_1 < \dots < i_{d+a}$ are chosen from $[M(d+a)]$. 

Indeed, any decomposable element is in the $\GL$-orbit of this element, but the set of $f \ast_g n$ with $f$, $g$ and $n$ varying is $\GL$-invariant. Accordingly, if every equation in this set vanishes on $e_{i_1} \wedge \cdots \wedge e_{i_{d+a}}$, they will vanish on every decomposable element and so will be in $\cS_{M}$.

However, it must be the case that the set of equations $f \ast_g n$ for any $f$, $g$ and $n$ vanish on $e_{i_1} \wedge \cdots \wedge e_{i_{d+a}}$, because only $e_{i_1}^\ast \wedge \cdots \wedge e_{i_{d+a}}^\ast$ does not vanish on $e_{i_1} \wedge \cdots \wedge e_{i_{d+a}}$. As mentioned above, $f \ast_g n$ is a sum of products of terms of the form $x_Ix_K$ where $I \not= K$, so one of $x_I$ or $x_J$ will vanish. This means every term in the sum vanishes, i.e. $f \ast_g n$ vanishes for any choice of $f$ and $n$.

The case for $f \cdot n$ is clearer because we just multiply $f$ and $n$ in $\cP_\Sigma$. By definition $f$ will vanish on any decomposable element and so $f \cdot n$ will vanish as well.
\end{proof}

We also notice that $\cP_\Sigma$ has a natural comultiplication $\Delta$ defined in the following way. If $v_1 \cdots v_n \in \Sym^n(\wedge^d \bk^{Md})$, then
\[
\Delta(v_1\cdots v_n) = \sum_{S \subseteq [n]} v_S \otimes v_{[n]\setminus S},
\]
where $v_S = \prod_{i \in S} v_i \in \Sym^{|S|}(\bigwedge^d \bk^{Md})$. Less explicitly, but still importantly, this is defined in the usual way by letting $w \mapsto w \otimes 1 + 1 \otimes w$ when $w \in \Sym^1(\bigwedge^d \bk^{Md})$ and extending uniquely while requiring that $\Delta \colon \cP_\Sigma \to \cP_\Sigma \otimes \cP_\Sigma$ be an algebra homomorphism.

We will now show that Lemma \ref{coinvideal} implies $J$ is a di-ideal. To do this, we need the following results.

\begin{proposition} \label{bigradedisomorphism}
The symmetrization map $\cG \colon \cP_\Sigma \to \cP^\Sigma$ is an isomorphism of bigraded bialgebras under the $\cdot$ product. More precisely, the following two diagrams commute:
\[
\begin{tikzcd}
\cP_\Sigma \otimes \cP_\Sigma \ar{r}[above]{\cdot} \ar{d}[left]{\cG \otimes \cG} & \cP_\Sigma \ar[d,"\cG"]\\
\cP^\Sigma \otimes \cP^\Sigma \ar{r}[above]{\cdot} & \cP^\Sigma
\end{tikzcd}
\qquad  \qquad
\begin{tikzcd}
\cP_\Sigma \ar{r}[above]{\Delta} \ar{d}[left]{\cG} & \cP_\Sigma \otimes \cP_\Sigma \ar[d,"\cG \otimes \cG"]\\
\cP^\Sigma \ar{r}[above]{\Delta} & \cP^\Sigma \otimes \cP^\Sigma
\end{tikzcd}
\]
\end{proposition}

\begin{proof}
This is the same proof as in \cite[Proposition 3.8]{Sa1}. 
%
%
\end{proof}

With this proposition, we can prove the following important fact:

\begin{proposition} \label{invariantideal}
If $\cJ \subseteq \cP_\Sigma$ is closed under $\ast_g$ and $\cdot$ in $\cP_\Sigma$, then $\cJ$ is a di-ideal. 
\end{proposition}

\begin{proof}
Proposition \ref{bigradedisomorphism} shows that if $\cJ \subset \cB_\Sigma$ is an ideal under $\cdot$, the same is true for $\cG(\cJ) \subseteq \cP^\Sigma$. 

It remains to show $\cG(\cJ)$ is closed under the $\ast_g$ product in $\cP^\Sigma$. We have $f \ast_g n \in \cJ$ for all monomials $n$ and $f \in \cJ$. In $\cP_\Sigma$, we defined
\[
f \ast_g n = \cG^{-1}(\cG(f) \ast_g \cG(n)),
\]
so
\[
\cG(f \ast_g n) = \cG(f) \ast_g \cG(n).
\]
As $\cG$ is a linear isomorphism, we can write any $n' \in \cP^\Sigma$ as $\cG(n')$ for some $n' \in \cP_\Sigma$. To see $\cG(\cJ)$ is closed under $\ast_g$, we just need to check $\cG(f) \ast_g n' \in \cG(\cJ)$ for any $n' \in \cP^\Sigma$. Find $n \in \cP_\Sigma$ such that $\cG(n) = n'$, then we know
\[
f \ast_g n \in \cJ,
\]
hence
\[
\cG(f) \ast_g \cG(n) \in \cG(\cJ),
\]
but $\cG(n) = n'$.
\end{proof}

\begin{theorem} \label{diideal}
For any fixed $M \geq 2$, $\cS_M$, is a di-ideal.
\end{theorem}

\begin{proof}
Combine Lemma \ref{coinvideal} with Propositions \ref{bigradedisomorphism} and \ref{invariantideal}.
\end{proof}

\section{Joins and Secants} \label{Joins and Secants}
Let $V$ be a vector space and $\Sym(V)$ be its symmetric algebra. Given ideals $I, J \subset \Sym(V)$, their {\bf join} $I \star J$ is the kernel of 
\[
\Sym(V) \xrightarrow{\Delta} \Sym(V) \otimes \Sym(V) \to \Sym(V)/I \otimes \Sym(V)/J,
\]
where the first map is the standard comultiplication. Note that $\star$ is an associative and commutative operation since $\Delta$ is coassociative and cocommutative. Set $I^{\star 0} = I$ and $I^{\star r} = I \star I^{\star(r-1)}$ for $r > 0$.

\begin{proposition}\cite[Proposition 4.1]{Sa1}
Assume $\bk$ is an algebraically closed field. If $I$ and $J$ are radical ideals, then $I \star J$ is a radical ideal. If $I$ and $J$ are prime ideals, then $I \star J$ is a prime ideal.
\end{proposition}

These definitions make sense for ideals $\cI, \cJ \subseteq \cP_\Sigma$, so we can define the join $\cI \star \cJ$. To be precise, $(\cI \star \cJ)_{d,n}$ is the kernel of the map
\[
(\cP_\Sigma)_{d,n} \xrightarrow{\Delta} \bigoplus_{i=0}^n (\cP_\Sigma/\cI)_{d,i} \otimes (\cP_\Sigma/\cJ)_{d,n-i}.
\]
Since $\cG$ is compatible with $\Delta$ (Proposition \ref{bigradedisomorphism}), we deduce that
\[
\cG(\cI \star \cJ) = \cG(\cI) \star \cG(\cJ).
\]
\begin{proposition} \label{secantideal}
If $\cI,\cJ \subseteq \cP_\Sigma$ are di-ideals, then $\cI \star \cJ$ is a di-ideal.
\end{proposition}

\begin{proof}
Pick $\bv \in \cG(\cI \star \cJ)$. By definition, $\bv$ is in the kernel of the map
\[
\Delta \colon \cP^\Sigma \to \cP^\Sigma/\cG(\cI) \otimes \cP^\Sigma/ \cG(\cJ).
\]
Since both $\cG(\cI)$ and $\cG(\cJ)$ are ideals under $\ast_g$ via Proposition \ref{invariantideal}, it gives a well-defined multiplication on $\cP^\Sigma/\cG(\cI) \otimes \cP^\Sigma/ \cG(\cJ)$. By Lemma \ref{comultproduct}, given $\bx \in \cP^\Sigma$, we have $\Delta(\bx \ast_g \bv) = \Delta(\bx) \ast_g \Delta(\bv)$. But $\Delta(\bv) = 0$, so $\bx \ast_g \bv \in \cG(\cI \star \cJ)$.
\end{proof}

Let $\bV$ be a vector space. A subscheme $X \subseteq \bV$ is {\bf conical} if its defining ideal $I_X$ is homogeneous. The {\bf $r$th secant scheme} of $X$ is the subscheme of $\bV$ defined by the ideal $I_X^{\star r}$. We wish to consider secant varieties of Pl\"ucker embeddings, we make the following definition,

\begin{definition}
For any fixed $r \geq 0$ and $M \geq 2$, let $\cS_{M}(r) = (\cS_{M})^{\star r}$. Where $\cS_M$ is defined in Definition \ref{SumDef}.
\end{definition}

An immediate corollary of what we have just shown

\begin{corollary} \label{rthsecantideal}
$\cS_M(r)$ is a di-ideal for any $r \geq 0$ and $M \geq 2$.
\end{corollary}

\begin{proof}
This follows immediately from combining Proposition \ref{secantideal} and Theorem \ref{diideal}.
\end{proof}

\section{Proof of Theorem \ref{TheoremA}} \label{5}
%

Before we can get to the main result, we need the following lemma

\begin{lemma} \label{Mbound}
For any fixed $r \geq 0$, to bound the degrees of the ideal generators of $\cS_{M}(r)$ for any $M$, it suffices to consider $M = (r+2)$.
\end{lemma}

\begin{proof}
This follows immediately from \cite[Proposition 5.7]{MM}.  Rephrasing, in our case for a given Grassmannian $\Gr(d,V)$, the number of nonzero rows in $\lambda$ with $\bS_\lambda(V) = \bigwedge^d V$ is $d$. So from the quoted result, to bound the degrees of the ideal generators of the $r$th secant variety of a given Grassmannian it suffices to consider a vector space of dimension $(r+2)d$. It is then clear that to bound the degrees of the ideal generators of $\cS_{M}(r)$ it suffices to consider $\cS_{r+2}(r)$ in $\cP_{r+2}$.
\end{proof}

\begin{remark}
The result,  \cite[Proposition 5.7]{MM}, specifically says that if the degrees of the ideal generators of the $r$th secant variety of a given Grassmannian for a vector space of dimension $(r+2)d$ are bounded by $C$, then the same is true for the degrees of the ideal generators of the $r$th secant variety of a given Grassmannian for a vector space of higher dimension. One might be concerned that we are not considering vector spaces of dimension less than $(r+2)d$. However if we work with any lower dimensional vector space say of dimension $c < (r+2)d$, the ideal of the $r$th secant variety of $\Gr(d,c)$ is contained in the ideal of the $r$th secant variety of $\Gr(d,(r+2)d)$ by the exact argument seen in the proof of  \cite[Proposition 5.7]{MM}. So for any given $d$, we can fix the dimension of the vector spaces as $(r+2)d$ to bound the ideal generators of the $r$th secant variety of $\Gr(d,n)$ for any $n$.
\end{remark}
%

\noindent{\em Proof of Theorem \ref{TheoremA} and Corollary \ref{CorollaryA}.} From Lemma \ref{Mbound} it suffices to consider $\Gr(d,(r+2)d)$. The ideal of ${\rm Sec}_r(\rho(\Gr(d,(r+2)d)))$ is contained in $\cS_{r+2}(r)$, where again $\rho$ is the Pl\"ucker embedding. In particular, $(\cS_{r+2}(r))_{d,n}$ is precisely the space of degree $n$ polynomials in the ideal of ${\rm Sec}_r(\rho(\Gr(d,(r+2)d)))$. 

Corollary \ref{rthsecantideal} implies $\cS_{r+2}(r)$ is a di-ideal. By Proposition \ref{invfingen}, $\cG(\cS_{r+2}(r))$ is generated by finitely many elements $f_1,\dots,f_N$ under $\cdot$ and $\ast_g$. Every element of $\cG(\cS_{r+2}(r))$ can be written as a linear combination of elements of the form $h \cdot (f_i \ast_g a)$ and so for fixed $d$ a set of generators for  ${\rm Sec}_r(\rho(\Gr(d,(r+2)d)))$ can be taken to be the set of all $\cG^{-1}(f_i \ast_g a)$ such that $f_i \ast_g a \in \cP_{d,n}^\Sigma$ for some $n$. The degree of $f_i \ast_g a$ is the same as that of $f_i$ (if $f \in \cP^\Sigma_{d,n}$ then its degree is $n$). So we can take $C(r) = \max(\deg(f_1),\dots,\deg(f_N))$.

The proof of Corollary \ref{CorollaryA} follows immediately from the above. \qed
%
%

\begin{remark}
In the case where $r=0$, this theorem tells us that the homogeneous ideal of the Pl\"ucker image of all Grassmannians can be generated by finitely many polynomials of a finite degree bounded by $C(0)$ under the operations $\ast_g$ and $\cdot$. In this case we know $C(0) = 2$ and Example \ref{pluckerfingen} shows that all Pl\"ucker equations can be obtained from the Klein quadric.
\end{remark}

\section{The Pl\"ucker Category} \label{6}
In this section, we will translate the above work into the language of functor categories in the spirit of \cite{Sa2}. Once we transition to this language, we can study free resolutions of secant ideals of Pl\"ucker embedded Grassmannians. In particular, we will show that the $i$th syzygy module of the coordinate ring of the $r$th secant variety of the Pl\"ucker embedded $\Gr(d,n)$ (whose space of generators is the $i$th Tor group with the residue field) is generated in bounded degree with bound independent of $d$ and $n$. The case $i = 1$ corresponds to the above results.

Let $\bk$ be a commutative ring and fix $M \geq 0$. Recall that $\cP_{d,n} = \Sym^n(\bigwedge^{d} \bk^{Md})$. We will now encode the morphisms from $\cP_{d,m}$ to $\cP_{e,n}$ as the space of morphisms from an objects $(d,m)$ to another object $(e,n)$ in the abstract category $\bk\cG_M$. The operations $\cdot$ and $\ast_g$ tell us how to do this when $d = e$ or $m= n$ respectively. More explicitly, when $d = e$ an operation $\cP_{d,m}$ to $\cP_{d,n}$ is given by a partition $\sigma$ of $[n]$ and an element of $\cP_{d,n-m}$. A basis for these operations can be encoded by an order preserving injection $[m] \to [n]$ together with a list of monomials. 

When $m = n$, an operation $\cP_{d,n}$ to $\cP_{e,n}$ consists of a choice of an element of $\cP_{e-d,n}$ as well as a map $g \in {\rm inc}(\bN)$ with $g([d]) \subset [e]$. It has a basis given by the monomials, which are represented by an ordered list of $n$ monomials in $\bigwedge^{e-d} \bk^{M(e-d)}$. Once again, we prefer to represent these lists of monomials by reading lists, denote the poset of readings lists by ${\bf RL}$, and the poset of readings lists with $n$ entries of size $e$ by ${\bf RL}_{n,e}$. Explicitly, if $S \in {\bf RL}_{n,e}$, then $S = (S^1,\dots,S^n)$ with $|S^i| = e$. Where the readings lists are defined above in \S2, below Corollary \ref{basis}. In particular, each $S^i$ consists of distinct numbers selected from $[Me]$.

When $d \not= e$ and $m \not= n$, it is harder to describe a basis for the space of operations. Given a map $\alpha \colon [d] \to [e]$, suppose $[e] \setminus \alpha([d]) = \{a_1,\dots,a_{e-d}\}$, define $\alpha^c \colon [e-d] \to [e]$ as $\alpha^c(i) = a_i$. We call this the {\bf complement} of $\alpha$.

\begin{definition}
Define the {\bf Pl\"ucker category} $\cG_M$ as follows. The objects of $\cG_M$ are pairs $(d,m) \in \bZ^2_{\geq 0}$ and a morphism $\alpha: (d,m) \to (e,n)$ consists of the following data:
\begin{itemize}
\item An order-preserving injection $\alpha_1: [m] \to [n]$.
\item A function $\alpha_2: [n] \setminus \alpha_1([m]) \to \RL_{1,e}$
\item A function $\alpha_3: [m] \to \RL_{1,e-d}$
\item An order-preserving injection $\alpha_4 \colon [d] \to [e]$.
\end{itemize}
In particular, $\hom_{\cG_M}((d,m),(e,n)) = \emptyset$ if $d > e$. Given another morphism $\beta \colon (e,n) \to (f,p)$, the composition $\beta \circ \alpha = \gamma \colon (d,m) \to (f,p)$ is defined by
\begin{itemize}
\item $\gamma_1 = \beta_1 \circ \alpha_1$.
\item $\gamma_2 \colon [p] \setminus \gamma_1([m]) \to \RL_{1,f}$ is defined by:
\begin{itemize}
\item if $i \in [p] \setminus \beta_1([n])$, then $\gamma_2(i) = \beta_2(i)$, and
\item if $i \in \beta_1([n] \setminus \alpha_1([m]))$, then $\gamma_2(i) = \beta_4(\alpha_2(i')) + \beta_4^c(\beta_3(i'))$ where $i'$ is the unique preimage of $i$ under $\beta_1$.
\end{itemize}
\item $\gamma_3 \colon [m] \to \RL_{1,f-d}$ is defined by $\gamma_3(i) = (\gamma_4^c)^{-1}(\beta_4\alpha_4^c\alpha_3(i) + \beta_4^c\beta_3(\alpha_1(i)))$.
\item $\gamma_4 = \beta_4 \circ \alpha_4$.
\end{itemize}
\end{definition}
When $d=e$, the functions $\alpha_3$ and $\alpha_4$ are superfluous and the pair $(\alpha_1,\alpha_2)$ encodes an operation as discussed above. Similarly, when $n = m$, the functions $\alpha_1$ and $\alpha_2$ are superfluous, and $\alpha_3$ also encodes an operation as discussed above. 

\begin{remark} \label{remark1}
Each of these maps $\alpha_2$ and $\alpha_3$ can be represented by reading lists in $\RL_{n-m,e}$ and $\RL_{m,e-d}$ respectively. To explicitly see this, $\alpha_3$ can be represented as $S_{m,e-d}$ with $S^i$ exactly the image of $i$. We do not take this perspective for ease of composition in the above definition. However, taking the reading list perspective will be important in Proposition \ref{Ggrobner}.
\end{remark}

\begin{lemma}
Composition as defined above is associative.
\end{lemma}

\begin{proof}
Suppose we are give three morphisms
\[
(d,m) \xrightarrow{\alpha} (e,n) \xrightarrow{\beta} (f,p) \xrightarrow{\gamma} (g,q).
\]
We will verify that all three components of both ways of interpreting $\gamma\beta\alpha$ are the same.
\begin{itemize}
\item The associativity of the first and fourth maps follows by the associativity of function composition.
\item Consider $[q] \setminus(\gamma_1\beta_1\alpha_1([m])) \to \RL_{1,g}$.
\begin{itemize}
\item If $i \in [q] \setminus \gamma_1([p])$ then $i \mapsto \gamma_2(i)$ under both compositions.
\item If $i \in \gamma_1([p] \setminus \beta_1([n]))$, let $i'$ be the unique premiage of $i$ under $\gamma_1$ and let $i''$ be the unique preimage of $i'$ under $\beta_1$.\\
Under $\gamma(\beta\alpha)$, we have $i \mapsto \gamma_4((\beta\alpha)_2(i')) + \gamma_4^c((\gamma)_3(i')) = \gamma_4\beta_2(i') + \gamma_4^c\gamma_3(i')$\\
Under $(\gamma\beta)\alpha$, we have $i \mapsto (\gamma\beta)_2(i') = \gamma_4\beta_2(i') + \gamma_4^c\gamma_3(i')$. This is because $i \in [q] \setminus \gamma_1\beta_1([n])$ by assumption. 
\item If $i \in \gamma_1\beta_1([n] \setminus \alpha_1([m]))$, let $i'$ be the unique preimage of $i$ under $\gamma_1$ and let $i''$ be the unique preimage of $i'$ under $\beta_1$.\\
Under $\gamma(\beta\alpha)$, we have $i \mapsto \gamma_4((\beta\alpha)_2(i')) + \gamma_4^c\gamma_3(i') = \gamma_4\beta_4\alpha_2(i'') +\alpha_4\beta_4^c\beta_3(i'') + \gamma_4^c \gamma_3(i')$.\\
Under $(\gamma\beta)\alpha$, we have $i \mapsto (\gamma\beta)_4(\alpha_2(i'')) + (\gamma\beta)_4^c ((\gamma\beta)_3(i'')) = \gamma_4\beta_4\alpha_2(i'') + \alpha_4\beta_4^c\beta_3(i'') + \gamma_4^c\gamma_3(\beta_1(i''))$.\\
But $\beta_1(i'') = i'$, so these are equal.
\end{itemize}
\item Now consider the map $[m] \to \RL_{1,g-d}$.
\begin{itemize}
\item Under $\gamma(\beta\alpha)$, we have $i \mapsto ((\gamma_4 \beta_4 \alpha_4)^c)^{-1}(\gamma_4(\beta_4 \alpha_4)^c (\beta\alpha)_3(i) + \gamma_4^c \gamma_3(\beta\alpha)_1(i)) = ((\gamma_4 \beta_4 \alpha_4)^c)^{-1} \\ (\gamma_4\beta_4\alpha_4^c\alpha_3(i) + \gamma_4\beta_4^c\beta_3\alpha_1(i)+\gamma_4^c\gamma_3\beta_1\alpha_1(i))$.
\item Under $(\gamma\beta)\alpha$, we have $i \mapsto ((\gamma_4 \beta_4 \alpha_4)^c)^{-1}((\gamma\beta)_4 \alpha_4^c \alpha_3(i) + (\gamma_4\beta_4)^c(\gamma\beta)_3(\alpha_1(i))) =((\gamma_4 \beta_4 \alpha_4)^c)^{-1} \\(\gamma_4\beta_4^c\beta_3\alpha_1(i)+\gamma_4^c\gamma_3\beta_1\alpha_1(i) + \gamma_4\beta_4\alpha_4^c\alpha_3(i))$.
\end{itemize}
\end{itemize}
We have equality in each case. As these are all the ways the components interact, this implies associativity.
\end{proof}

Let $\bk \cG_M$ be the $\bk$-linearization of $\cG_M$, i.e., $\hom_{\bk \cG_M}(x,y) = \bk[\hom_{\cG_M}(x,y)]$. A {\bf $\bk\cG_M$-module} is a functor form $\cG_M$ to the category of $\bk$-modules. Equivalently, a $\bk\cG_M$-modules is a $\bk$-linear functor from $\bk\cG_M$ to the category of $\bk$-modules. Morphisms of $\bk\cG_M$-modules are natural transformations, and $\bk\cG_M$-modules form an abelian category where submodules, kernels, cokernels, etc. are computed component pointwise.

Given $(d,m) \in \bZ^2_{\geq 0}$, define a $\bk \cG_M$-module $P_{d,m}$ by
\[
P_{d,m}(e,n) = \bk[\hom_{\cG_M}((d,m),(e,n))].
\]
This is the principal projective $\bk\cG_M$-module generated in bidegree $(d,m)$, and they give a set of projective generators for the category of $\bk\cG_M$-modules. That is, every $\bk\cG_M$-modules is a quotient of a direct sum of principal projectives. For further exposition on principal projectives we refer the reader to \cite[\S3.1]{SS1}. Then $P_{d,m}(e,n)$ is the space of operations from $\cP_{d,m}$ to $\cP_{e,n}$ which we discussed at the beginning of the section, so $P_{d,m}$ is a $\cP$-module freely generated in bidegree $(d,m)$.

To emphasize the category we may sometimes write $P^{\cG_M}_{d,m}$. With these definition we can now make sense of what it means for modules to be finitely generated. A $\bk\cG_M$-modules $N$ is {\bf finitely generated} if there is a surjection
\[
\bigoplus_{i=1}^g P_{d_i,m_i} \to N \to 0,
\]
with $g$ finite. A $\bk\cG_M$-modules is {\bf noetherian} if all of its submodules are finitely generated. For a definition of a Gr\"obner category, see \cite[Definition 4.3.1]{SS1}. We only need this definition for the next result and it is lengthy, so we choose to omit it so as not to distract.

\begin{proposition} \label{Ggrobner}
$\cG_M$ is a Gr\"obner category. In particular, if $\bk$ is noetherian, then every finitely generated $\bk \cG_M$ module is noetherian.
\end{proposition}

\begin{proof}
We will use \cite[Theorem 4.3.2]{SS1}. Fix $(d,m) \in \bZ_{\geq 0}^2$. Let
\[
\Sigma = \RL \times \RL \times  \bZ_{\geq 0}^m \times  \bZ_{\geq 0}^d
\]
By Dickson's Lemma and Theorem \ref{RLNoeth}, the finite product of noetherian posets is also noetherian with the componentwise order. Hence $\Sigma$ is noetherian.

From Remark \ref{remark1} we can associate to each $\alpha_2$ and $\alpha_3$ a RL, $S_{\alpha_i}$, for $i = 2,3$. Given a morphism $\alpha \colon (d,m) \to (e,n)$ encode it as $w(\alpha) \in \Sigma$
\[
w(\alpha) = (S_{\alpha_2},S_{\alpha_3},{\rm im}(\alpha_1),{\rm im}(\alpha_4)).
\]
We can recover $\alpha$ from $w(\alpha)$, so this is an injection. Define $\alpha \leq \gamma$ if there exists some $\beta$ such that $\gamma = \beta \circ \alpha$. Then, it follows from the definition of composition that the set of morphisms $\alpha \colon (d,m) \to (e,n)$ with $(d,m)$ fixed and $(e,n)$ varying is naturally a subposet of $\Sigma$, i.e. $\alpha \leq \alpha'$ if and only if $w(\alpha) \leq w(\alpha')$. Since noetherianity is inherited by subposets, we conclude that this partial order on morphisms with source $(d,m)$ is noetherian. 

It remains to prove that the set of morphisms with source $(d,m)$ is orderable, i.e., for each $(e,n)$ there exists a total ordering on the set of morphisms $(d,m) \to (e,n)$ so that for any $\beta \colon (e,n) \to (f,p)$, we have $\alpha < \alpha'$ implies that $\beta \alpha < \beta \alpha'$. To do this first put the lexicographic order on $\RL$. That is given two RLs, $S_{d,n} = (S^1,\dots,S^n)$ and $T_{e,m} = (T^1,\dots,T^m)$ we say $S_{d,n} \leq T_{e,m}$ if $n < m$ or if $n = m$ and $d < e$, if $m = n$ and $d = e$ we compare the lists in $(\bZ^d)^n$ using the natural lexicographic order described in $\S2$. In particular, we first compare $S^1$ and $T^1$ lexicographically, if they are equal we consider $S^2$ and $T^2$, etc. 

This defines a total order on $\RL$. Now put a lexicographic order on $\bZ_{\geq 0}^m$ and $\bZ_{\geq 0}^d$ in the natural way. Totally order $\Sigma$ by declaring all of the elements of the first $\RL$ to be larger than the second $\RL$ which is larger than $\bZ_{\geq 0}^m$ which is larger than $\bZ_{\geq 0}^d$. This is just another lexicographic order. This orders $\Sigma$, which in turn gives the desired ordering.
%
\end{proof}

This proves $\cG_M$ is Gr\"obner, in particular this also implies Theorem \ref{monfingen}.

\subsection{Symmetrized versions.}
In $\bk\cG_M$, the space of morphisms $(0,0) \to (d,m)$ is identified with the tensor power $(\bigwedge^d \bk^{Md})^{\otimes m}$. For our applications, we need symmetric powers, $\Sym^m(\bigwedge^d \bk^{Md})$, so we now define symmetrized versions of the Gr\"obner category $\bk\cG_M$.

\begin{definition}
Given $\alpha \colon (d,m) \to (e,n)$ and $\sigma \in \Sigma_n$, there is a unique $\tau \in \Sigma_m$ so that $\sigma\alpha_1\tau^{-1}$ is order-preserving; we refer to $\tau$ as the permutation induced by $\sigma$ with respect to $\alpha_1$. Define $\sigma(\alpha)$ by
\begin{itemize}
\item $\sigma(\alpha)_1 = \sigma\alpha_1\tau^{-1}$,
\item $\sigma(\alpha)_2 = \alpha_2 \sigma^{-1}$,
\item $\sigma(\alpha)_3 = \alpha_3\tau^{-1}$,
\item $\sigma(\alpha)_4 = \alpha_4$.
\end{itemize}
This defines an action of $\Sigma_n$ on $\hom_{\cG_M}((d,m),(e,n))$, and we set
\[
\hom_{\bk\cG_M^\Sigma}((d,m),(e,n)) = \bk[\hom_{\cG_M}((d,m),(e,n))]^{\Sigma_n}
\]
where the superscript denotes taking invariants.
\end{definition}

\begin{lemma} \label{symmetry}
Given $\alpha \colon (d,m) \to (e,n)$ and $\beta \colon (e,n) \to (f,p)$, and $\sigma \in \Sigma_p$, we have $\sigma(\beta \circ \alpha) = \sigma(\beta) \circ \tau(\alpha)$ where $\tau \in \Sigma_n$ is the permutation induced by $\sigma$ with respect to $\beta_1$. In particular, $\bk\cG_M^\Sigma$ is a $\bk$-linear subcategory of $\bk\cG_M$.
\end{lemma}

\begin{proof}
Let $\rho \in \Sigma_\ell$ be the permutation induced by $\tau$ with respect to $\alpha_1$. Then $(\sigma\beta_1\tau^{-1})(\tau\alpha_1\rho^{-1})$ is order preserving, so $\rho$ is also the permutation induced by $\sigma$ with respect to $\beta_1\alpha_1$. Hence $\sigma(\beta\alpha)_1 = \sigma(\beta)_1\tau(\alpha)_1$.

Next, we show that $\sigma(\beta\alpha)_2 = (\sigma(\beta)\tau(\alpha))_2$. If $i \in [p] \setminus \sigma(\beta)_1([n])$, then
\[
\sigma(\beta\alpha)_2(i) = (\beta\alpha_2)\sigma^{-1}(i) = \beta_2\sigma^{-1}(i) = \sigma(\beta)_2(i) = (\sigma(\beta)\tau(\alpha))_2(i).
\]
Else, $i \in \sigma(\beta)_1([n] \setminus \tau(\alpha)_1([m]))$, let $i'$ be the unique preimage of $i$ under $\sigma \beta_1\tau^{-1}$. Then $\tau^{-1}(i')$ is the unique preimage of $\sigma^{-1}(i)$ under $\beta_1$, and we have
\begin{align*}
\sigma(\beta\alpha)_2(i) &= (\beta\alpha)_2\sigma^{-1}(i) = \beta_4(\alpha_2(\tau^{-1}(i'))) + \beta_4^c(\beta_3(\tau^{-1}(i')))\\
&= \beta_4(\tau(\alpha)_2(i')) + \beta_4^c(\sigma(\beta)_3(i')) = (\sigma(\beta)\tau(\alpha))_2(i).
\end{align*}
Now, we show that $\sigma(\beta\alpha)_3 = (\sigma(\beta)\tau(\alpha))_3$. For $i \in [\ell]$, we have
\begin{align*}
(\sigma(\beta)\tau(\alpha))_3(i) &=((\sigma(\beta)_4\tau(\alpha)_4)^c)^{-1}(\sigma(\beta)_4\tau(\alpha)_4^c\tau(\alpha)_3(i) + \sigma(\beta)_4^c \sigma(\beta)_3\tau(\alpha)_1(i))\\
&=((\sigma(\beta)_4\tau(\alpha)_4)^c)^{-1}(\beta_4\alpha_4^c\alpha_3\rho^{-1}(i) + \beta_4^c\beta_3\tau^{-1}\tau\alpha_1\rho^{-1}(i))\\
&=((\sigma(\beta)_4\tau(\alpha)_4)^c)^{-1}(\beta_4\alpha_4^c\alpha_3 + \beta_4^c\beta_3\alpha_1)(\rho^{-1}(i))\\
&=(\beta\alpha)_3\rho^{-1}(i)\\
&= \sigma(\beta\alpha)_3.
\end{align*}
Finally, we show that $\sigma(\beta\alpha)_4 = (\sigma(\beta)\tau(\alpha))_4$. This is clear because $\sigma$ acts trivially on this map, so
\[
\sigma(\beta\alpha)_4 = \beta_4 \alpha_4 = \sigma(\beta)_4\tau(\alpha)_4 = (\sigma(\beta)\tau(\alpha))_4.
\]
\end{proof}

A $\bk\cG_M^\Sigma$-module is a $\bk$-linear functor from $\bk\cG_M^\Sigma$ to the category of $\bk$-modules. For each $(d,m)$, the principal projective $\bk\cG_M^\Sigma$-module is defined by 
\[
P_{d,m}^{\bk\cG_M^\Sigma}(e,n) = \hom_{\bk\cG_M^\Sigma}((d,m),(e,n)),
\]
and we say that a $\bk\cG_M^\Sigma$-module $N$ is finitely generated if there is a surjection
\[
\bigoplus_{i=1}^g P_{d_i,m_i}^{\bk\cG_M^\Sigma} \to N \to 0
\]
with $g$ finite.

\begin{proposition} \label{invnoeth}
If $\bk$ contains a field of characteristic $0$, then every finitely generated $\bk\cG_M^\Sigma$-module is noetherian.
\end{proposition}

\begin{proof}
Set $P_{d,m} = P_{d,m}^{\cG_M}$ and $Q_{d,m} = P_{d,m}^{\bk\cG_M^\Sigma}$; we have a natural inclusion $Q_{d,m}(e,n) \subseteq P_{d,m}(e,n)$ for all $(e,n)$. Given a $\bk\cG_M^\Sigma$-submodule $M$ of $Q_{d,m}$, let $N$ be the $\cG_M$-submodule of $P_{d,m}$ that it generates. Given a list of generators of $N$ coming from $M$, Proposition \ref{Ggrobner} shows $P_{d,m}$ is noetherian and so some finite subset $\gamma_1,\dots,\gamma_g$ of them already generates $N$. Let $\pi$ be the symmetrization map
\begin{align*}
\bk[\hom_{\cG_M}((d,m),(e,n))] &\to \bk[\hom_{\cG_M}((d,m),(e,n))]^{\Sigma_n}\\
\alpha &\mapsto \tfrac{1}{n!} \sum_{\sigma \in \Sigma_n} \sigma(\alpha).
\end{align*}
If $\alpha \in \bk[\hom_{\cG_M}((d,m),(e,n))]^{\Sigma_n}$, then $\pi(\alpha) = \alpha$; given $\beta \in \bk[\hom_{\cG_M}((e,n),(f,p))]$, then $\pi(\beta\alpha) = \pi(\beta)\alpha$ by Lemma \ref{symmetry}.

Given any element $\gamma$ of $M$, we have an expression $\gamma = \sum_i \delta_i\gamma_i$, where $\delta_i \in \bk\cG_M$. So, applying $\pi$, we get $\gamma = \pi(\gamma) = \sum_i \pi(\delta_i)\gamma_i$, so $\gamma_1,\dots,\gamma_g$ also generate $M$ as a $\bk\cG_M^\Sigma$-module. In particular, the principal projectives of $\bk\cG_M^\Sigma$ are noetherian, so the same is true for any finitely generated module.
\end{proof}

Now assume that $\bk$ contains a field of characteristic $0$. We define the {\bf symmetrized Pl\"ucker category} $\sG_M = (\bk\cG_M)_\Sigma$ as follows. First, set
\[
\hom_{\sG_M}((d,m),(e,n)) = \bk[\hom_{\cG_M}((d,m),(e,n))]_{\Sigma_n}
\]
where the subscript denotes coinvariants under $\Sigma_n$. As in \S3, we have an isomorphism
\begin{align*}
\bk[\hom_{\cG_M}((d,m),(e,n))]_{\Sigma_n} &\xrightarrow{\cong} \bk[\hom_{\cG_M}((d,m),(e,n))]^{\Sigma_n}\\
\alpha &\mapsto \tfrac{1}{n!} \sum_{\sigma \in \Sigma_n} \sigma(\alpha),
\end{align*}
and as above we use this to transfer the $\bk$-linear category structure from $\bk\cG_M^\Sigma$ to $\sG_M$. Now we notice that $\hom_{\sG_M}((0,0),(d,m))$ is identified with $\Sym^m(\bigwedge^d \bk^{(r+2)d})$, which was our goal. This isomorphism in combination with Proposition \ref{invnoeth} give us the following:

\begin{proposition} \label{coinvnoeth}
Suppose $\bk$ is a field of characteristic $0$. Every finitely generated $\sG_M$-module is noetherian.
\end{proposition}

\begin{remark}
These definitions parallel the constructions in \S\ref{3}. In particular, we can identify $\cP^\Sigma$ and $\cP_\Sigma$ from this section with the principal projectives generated in degree $(0,0)$ in $\bk\cG_M^\Sigma$ and $\sG_M$ respectively. Furthermore, the notions of ideal and di-ideal translate to submodules in both cases. So Proposition \ref{invfingen} is a special case of Proposition \ref{invnoeth}.
\end{remark}

\section{Syzygies of Secant Ideals} \label{7}
In this section, $\bk$ is a field of characteristic $0$. For this section fix some $M \geq 0$. The principal projective $P_{0,0}$ in $\sG_M$ is the algebra $\cP_\Sigma$ from \S\ref{3} and each principal projective $P_{d,m}$ is a module over it. We use $\cP_\Sigma(-d,-m)$ to denote this module; by Proposition \ref{coinvnoeth} these are all noetherian modules.

In Definition \ref{SumDef}, we defined $\cS_M$ to be the sum of the Pl\"ucker ideals corresponding to $\Gr(d,Md)$ as $d \geq 0$ varies and $\cS_M(r) = (\cS_M)^{\star r}$. By Corollary \ref{rthsecantideal}, $\cS_{M}(r)$ is a $\sG_M$-submodule of $P_{0,0}$ for all $r$.

For $d$ fixed, $\bigoplus_m \cS_{M}(r)_{d,m}$ is an ideal in $\Sym(\bigwedge^d \bk^{Md})$. So we can define an alegbra
\[
\Sec_{d,r}(M) = \bigoplus_{m \geq 0} \Sym^{m}(\bigwedge^d \bk^{Md})/\cS_{M}(r)_{d,m}
\]
which is a quotient of $\Sym(\bigwedge^d \bk^{Md})$. Notice this is exactly the $r$th secant ideal of the Pl\"ucker-embedded $\Gr(d,Md)$. More generally, if $M$ is a $\sG_M$-module, then for $d$ fixed, $\bigoplus_m M_{d,m}$ is a $\Sym(\bigwedge^d \bk^{Md})$-module.

\begin{lemma} \label{free}
Fix $d,e,n$. Then
\[
\bigoplus_{m \geq 0} \cP_\Sigma(-e,-n)_{d,m}
\]
is a free $\Sym(\bigwedge^d \bk^{Md})$-module generated in degree $n$ whose rank is $\dim_k(\Sym^n(\bigwedge^{d-e} \bk^{M(d-e)}))$.
\end{lemma}

\begin{proof}
We have
\begin{align*}
\bigoplus_{m \geq 0} \cP_\Sigma(-e,-n)_{d,m} &= \bigoplus_{m \geq 0} (\Sym^n(\bigwedge^{d-e} \bk^{M(d-e)}) \otimes \Sym^{m-n}(\bigwedge^d \bk^{Md}))\\
&=\Sym^n(\bigwedge^{d-e} \bk^{M(d-e)}) \otimes \Sym(\bigwedge^d \bk^{Md})(-n).
\end{align*} 
As follows from the definitions, the action of $\Sym(\bigwedge^d \bk^{Md})$ on this space corresponds to the usual multiplication on  $\Sym(\bigwedge^d \bk^{Md})(-n)$.
\end{proof}

\begin{theorem}\label{syzygy}
There is a function $C_M(i,r)$, depending on $i,r,M$, but independent of $d$, such that 
\[
\Tor_i^{\Sym(\bigwedge^d \bk^{Md})}(\Sec_{d,r}(M),\bk)
\]
is concentrated in degrees $\leq C_M(i,r)$.
\end{theorem}

\begin{proof}
We know $\cS_{M}(r)$ is a finitely generated submodule of $\cP_\Sigma$, and hence has a projective resolution
\[
\cdots \to \bF_i \to \bF_{i-1} \to \cdots \to \bF_0,
\]
such that each $\bF_i$ is a finite direct sum of principal projective modules by Proposition \ref{coinvnoeth}. For $d$ fixed, we get an exact complex of $\Sym(\bigwedge^d \bk^{Md})$-modules
\[
\cdots \to \bigoplus_{m} (\bF_i)_{d,m} \to \bigoplus_{m} (\bF_{i-1})_{d,m} \to \cdots \to \bigoplus_{m} (\bF_0)_{d,m} \to \Sec_{d,r}(M) \to 0.
\]
If $\bF_i = \bigoplus_{j=1}^k \cP_{\Sigma}(-d_j,-m_j)$, then set $C_M(i,r) = \max(m_1,\dots,m_k)$. In particular, by Lemma \ref{free}, this gives a free resolution which can be used to compute $\Tor_i^{\Sym(\bigwedge^d \bk^{Md})}(\Sec_{d,r}(M),\bk)$ which we conclude is concentrated in degrees $\leq C_M(i,r)$.
\end{proof}

\begin{remark}
If we write $T_{i;d,r}(M) = \Tor_i^{\Sym(\bigwedge^d \bk^{Md})}(\Sec_{d,r}(M),\bk)$. As used above, this is $\bZ$-graded, so we denote the $m$th graded component by $T_{i;d,r}(M)_m$. For fixed $i.m,r$, we get a functor on the full subcategory $\bk\cG_M$ on objects of the form $(d,m)$ by
\[
(d,m) \mapsto T_{i;d,r}(M)_m.
\]
From the results above, we conclude that this is a finitely generated functor. In particular, as we allow $d$ to vary, this means that $T_{i;d,r}(M)_m$ is ``built out" of $T_{i;d',r}(M)_m$ where the $d'$ range over some finite list of integers. This can be thought of as the Pl\"ucker analogue of $\Delta$-modules from \cite{Sn}.
\end{remark}

As above, we would like to find a bound independent of $M$, i.e., independent of the chosen vector space for the Pl\"ucker embedded Grassmannian.

\begin{theorem} \label{Mindependent}
The function $C_M(i,r)$ is independent of $M$ once $M \geq r+1+i$. In particular, there is a bound $C(i,r)$ that works for all $M$ simultaneously.
\end{theorem}

\begin{proof}
First we notice that $\Sec_{d,r}(M)$ is a direct sum of Schur functors $\bS_\lambda((\bk^{Md})^\ast)$ with $\ell(\lambda) \leq d(r+1)$ by \cite[\S5.1]{MM}. The $i$th Tor module of $\Sec_{d,r}(M)$ is the $i$th homology of the Koszul complex of $\bigwedge^d (\bk^{Md})^\ast$ tensored (over $\Sym(\bigwedge^d (\bk^{Md})^\ast)$) with $\Sec_{d,r}(M)$, so is a subquotient of 
\[
\bigwedge^i \bigwedge^d (\bk^{Md})^\ast \subset (\bigwedge^d(\bk^{Md})^\ast )^{\otimes i}
\]
tensored (over $\bk$) with $\Sec_{d,r}(M)$. So all Schur functors $\bS_\mu((\bk^{Md})^\ast )$ that appear in the $i$th Tor module satisfy $\ell(\mu) \leq d(r+1) + di$ by the subadditivity of $\ell$. In particular, no information is lost by specializing to the case $Md = d(r+1+i)$ \cite[Corollary 9.1.3]{SS2}. So it suffices to take $M = r+1+i$.
\end{proof}

\begin{remark}
This is a special case of Lemma \ref{Mbound}. There, $i = 1$ and we can take $M = r+2$, so that $C_{r+2}(1,r) = C(r)$ from Theorem \ref{TheoremA}.
\end{remark}

{\em Proof of Theorem \ref{TheoremB}.} Combine Theorem \ref{syzygy} and Theorem \ref{Mindependent}.

\end{document}